\documentclass[12pt]{amsart}
\usepackage{amsmath, amssymb}
\usepackage[hypertexnames=false]{hyperref}
\usepackage{cleveref}
\usepackage{mathrsfs}
\usepackage{bm}
\usepackage{bbding}
\usepackage{textcomp}
\usepackage{multicol}
\usepackage{tasks}

\usepackage{empheq} 
\usepackage{fontawesome5}

\newcommand{\alphafootnote}[1]{%
	\begingroup
	\renewcommand{\thefootnote}{a}%
	\footnote{#1}%
	\endgroup
}

\usepackage[greek,english]{babel}

\usepackage[x11names]{xcolor}

\usepackage{graphicx}
\usepackage{amsthm,amscd}
\usepackage{calc}
\usepackage{mathtools}
\usepackage{CJK,fancyhdr,amscd}

\usepackage{anysize}
\usepackage{epsfig,enumerate}
\usepackage{latexsym}
\usepackage{indentfirst, latexsym}
\usepackage{graphics}
\usepackage{tikz-cd}
\usepackage{lineno}
\usepackage[all,poly,knot]{xy}

\usepackage{xpatch}
\usepackage[shortlabels]{enumitem}

\usepackage[T1]{fontenc}
\usepackage{textcomp}
\usepackage[utf8]{inputenc}

\hypersetup{
	colorlinks=true,
	linkcolor=Blue2,
	citecolor=SpringGreen4,
	urlcolor=Magenta2}

\renewcommand{\eqref}[1]{\textcolor{Red2}{\hyperref[#1]{(\ref*{#1})}}}

\xpatchcmd{\step}{%
	\normalfont\scshape\centering}{%
	\normalfont\scshape}{\typeout{Success}}{\typeout{Failure}}%

\allowdisplaybreaks[2]

\providecommand{\U}[1]{\protect\rule{.1in}{.1in}}

\numberwithin{equation}{section}

\newtheorem{theorem} {Theorem} [section]
\theoremstyle{plain}

\newtheorem{bigthm}{Theorem}

\newtheorem{bigcor}[bigthm]{Corollary}

\newtheorem{proposition}[theorem]{Proposition}

\newtheorem{lemma}[theorem]{Lemma}

\theoremstyle{definition}

\theoremstyle{plain}

\newtheorem{question}[theorem]{Question}
\newtheorem*{questionno}{Question}

\newtheorem{step}{Step}
\theoremstyle{definition}

\theoremstyle{definition}

\newtheorem{definition}  [theorem]     {Definition}

\theoremstyle{plain}

\theoremstyle{definition}
\newtheorem{remark}  [theorem]     {Remark}

\newcounter{specialremarkcounter}[section]
\renewcommand{\thespecialremarkcounter}{\thesection.\arabic{specialremarkcounter}}

\newtheoremstyle{specialremarkstyle} 
{}                          
{}                          
{\normalfont}               
{}                          
{\itshape}                  
{.}                         
{ }                         
{\textit{Remark \thespecialremarkcounter}} 

\theoremstyle{specialremarkstyle}

\newcommand{\Upomega}{\boldsymbol{\Omega}}

\newcommand{\w}{\wedge}

\newcommand{\dgla}{DGLA}

\newcommand{\LLL}{\mathrm{\mathbf L}}

\newcommand{\hooklongrightarrow}{\lhook\joinrel\longrightarrow}
\newcommand{\elemma}{\end{lemma}}
\newcommand{\blemma}{\begin{lemma}}

\newcommand{\p}{\partial}

\newcommand{\bp}{\bar\partial}

\newcommand{\End}{\operatorname{End}}

\newcommand{\bd}{\begin{enumerate} }
	\newcommand{\ed}{\end{enumerate}}

\newcommand{\tr}{\operatorname{tr}}

\DeclareMathAlphabet{\mathpzc}{OT1}{pzc}{m}{it}

\DeclareSymbolFont{extraup}{U}{zavm}{m}{n}
\DeclareMathSymbol{\varheart}{\mathalpha}{extraup}{86}
\DeclareMathSymbol{\vardiamond}{\mathalpha}{extraup}{87}

\def\p{\partial}

\def\mb{\mathbb}

\def\w{\wedge}

\def\l{\lrcorner}

\makeatletter

\newcommand{\subtitle}[1]{\gdef\@subtitle{#1}}
\newcommand{\@subtitle}{}

\renewcommand{\@settitle}{%
	\begin{center}
		\baselineskip 14\p@\relax
		\bfseries
		\uppercasenonmath\@title
		\@title\par
		\vspace{0.5em}
		{\normalfont\mdseries\upshape\sffamily\footnotesize
			\@subtitle\par}
	\end{center}
}

\makeatother

\begin{document}
	
	\title[Kuranishi space of joint deformations of $(X,E)$]
	{The Kuranishi space of joint deformations of Calabi--Yau
		manifolds and holomorphic vector bundles\\[0.5em]
		{\normalfont\small
			(with an Appendix joint with Hisashi Kasuya)}}
	
	\let\oldmaketitle\maketitle
	\renewcommand\maketitle{{\bfseries\boldmath\oldmaketitle}}
\begin{abstract}
	In this paper, we study joint deformations of pairs $(X,E)$, where
	$X$ is a compact complex manifold and $E\rightarrow X$
	is a holomorphic vector bundle.
	We prove that $(X,E)$ has unobstructed deformations
	if $X$ is a Fujiki manifold with torsion canonical bundle
	and $H^2(X,\End^0E)=0$, where $\End^0E$ denotes the
	trace-free endomorphism bundle.
	When the canonical bundle is trivial, the Fujiki assumption
	can be replaced by three weak $\partial\bar\partial$-conditions.
	We construct examples satisfying these conditions whose
	Fr\"olicher spectral sequences do not degenerate at $E_1$.
	These results provide non-K\"ahler extensions of the
	theorems of Li--Pan and Iacono--Manetti. 
	
	Without this vanishing assumption, we obtain both
	unobstructed and obstructed pairs.
	In Thomas's example, we prove unobstructedness of the pair,
	although the bundle has obstructed deformations with
	the manifold fixed.
	On the other hand, we prove that every strict projective
	Calabi--Yau manifold of dimension at least three admits
	a simple bundle with obstructed
	joint deformations.
	We also construct obstructed pairs with Hermitian
	flat bundles on complex tori.
	These results answer two questions raised by Felten.
\end{abstract}

	\author{Runze Zhang}

	\address{The Institute of Mathematical Sciences and Department of Mathematics, The Chinese University of
		Hong Kong, Shatin, N.T., Hong Kong}
	\email{runze.zhang@unice.fr}
	\urladdr{\href{https://sites.google.com/view/runzezhang}{https://sites.google.com/view/runzezhang}}


	\date{\today}
	


	\setcounter{tocdepth}{2}
	
	\maketitle

	\section{Introduction}
	\subsection{Background and motivation}For any compact complex manifold $X$, Kuranishi \cite{Kur65}
	constructed a semi-universal deformation which, in his words, parametrizes all
	sufficiently small deformations of $X$ in ``\textit{the most economical}'' way.
	Its analytic base is called the \textit{Kuranishi space} of $X$ and is
	denoted by $\textrm{Kur}(X)$. In general, $\textrm{Kur}(X)$ may be highly singular, even when
	$X$ is a surface \cite{Vak06}.

	\medskip
	Building on Bogomolov's  work \cite{Bog78} on holomorphic
	symplectic manifolds, Tian \cite{Tia87} and Todorov \cite{Tod89}
	established that every compact K\"ahler manifold $X$ with trivial
	canonical bundle has \textit{unobstructed deformations}, namely, that its
	Kuranishi space is smooth. This foundational result is now known as
	the 	\textit{Bogomolov--Tian--Todorov (BTT) theorem},  and its original proofs are
	differential-geometric in nature.
	Later, Ran, Kawamata and Fantechi--Manetti developed the
	$T^1$-lifting method \cite{Ran92,Kaw92,FM99}, which gives algebraic
	proofs of the BTT theorem, including the case where the canonical
	bundle is torsion; see also \cite{IM10}.
	
	\medskip
	Under the assumptions above, the K\"ahler condition can also be
	weakened. It suffices that $X$ satisfy the \textit{$\p\bp$-lemma} 	\cite{Huy05,Pop19}, meaning that
	$d$-exactness, $\p$-exactness, $\bp$-exactness and
	$\p\bp$-exactness are equivalent for every $d$-closed form of pure
	type (such an $X$ is called a \textit{$\p\bp$-manifold}). In fact, the weaker assumption that the
	Fr\"olicher spectral sequence degenerates at $E_1$  already suffices
	for unobstructedness \cite{KKP08,Iac17,ACRT18}; see also \cite{WX25} for a further weakening.

	\medskip
	\medskip
	In this paper, we study the Kuranishi space $\textrm{Kur}(X,E)$ of  joint deformations of a compact complex manifold $X$ and a holomorphic vector bundle $E$ over $X$. Such a deformation varies the complex structure of $X$ and the holomorphic structure of $E$ simultaneously.
	
	\medskip
	For  line bundles, Li--Pan and Iacono--Manetti independently proved  the following  remarkable results:
	\begin{theorem}[{\cite[Theorem 1.2]{LP19}}]\label{LP19}
		Let \(X\) be a smooth proper variety over \(\mathbb{C}\) with
		torsion canonical bundle, and let \(L\) be a line bundle on \(X\).
		Then the deformations of the pair \((X,L)\) are unobstructed.
	\end{theorem}
	
	\begin{theorem}[{\cite[Theorem 1.1]{IM21}}]\label{IM21}
		Let \(L\) be a line bundle on a smooth projective variety \(X\)
		defined over an algebraically closed field \(\mathbb{K}\) of
		characteristic \(0\). If \(X\) has trivial canonical bundle, then
		the DG-Lie algebra controlling the deformations of the pair
		\((X,L)\) is homotopy abelian. In particular, $(X,L)$ has unobstructed deformations.
	\end{theorem}
	
	\begin{enumerate}[{(1)}]
		\item \label{LPapproach}
		For smooth projective varieties, Li--Pan's geometric proof
		used the Beauville--Bogomolov decomposition
		\cite{Bog74,Bea83} to pass to a finite \'etale cover
		$$
		\widetilde X\simeq A\times Y\times Z,
		$$
		where $A$ is an abelian variety, and $Y$ and $Z$ are products
		of simply connected Calabi--Yau manifolds and irreducible
		holomorphic symplectic manifolds, respectively.
		They reduced the problem to proving unobstructedness for
		$(A,L_1)$, $(Y,L_2)$ and $(Z,L_3)$ separately, where the
		$L_i$ are the corresponding restrictions of the pull-back
		of $L$. The main difficulty was the abelian case, which
		they handled through the smoothness of the relevant
		Hodge locus and a formal extension theorem for line
		bundles \cite[\S\S6--7]{LP19}.
		
		\medskip
		For the initial decomposition step, the classical theorem
		was already available in the compact K\"ahler setting,
		with $A$ a complex torus. Analogous decompositions were
		established by Biswas--Cao--Dumitrescu--Guenancia for
		Moishezon manifolds in arbitrary dimension and Fujiki class $\mathscr{C}$
		manifolds of dimension at most four
		\cite[Theorem~A]{BCDG25}.
	Very recent advances in the K\"ahler MMP \cite{HX26,MZ26}
	now yield the arbitrary-dimensional Fujiki case;
	see  \cite[Theorem~9.5]{Gue26}.
		Here \textit{Moishezon} and \textit{Fujiki (class $\mathscr{C}$) manifolds}
		are compact complex manifolds bimeromorphic to projective
		and compact K\"ahler manifolds, respectively; both satisfy
		the $\p\bp$-lemma \cite[Corollary~5.23]{DGMS75}.
		
		\medskip
		\medskip
		\item \label{IMapproach}	Iacono--Manetti gave two approaches to the problem. The first used the $\mathbb P^1$-bundle $\mathbb P(\mathcal O_X\oplus L)$ with its two distinguished sections to reduce the problem to the unobstructedness of log Calabi--Yau pairs. Prior to the present work, such results were known in the
		projective setting \cite{KKP08,Iac15} and the compact K\"ahler
		setting \cite{Iac17,Ran17,LRW19,Wan18,Zha26}.   Felten adapted this construction
		to prove unobstructedness for proper log toroidal Calabi--Yau
		families equipped with a line bundle \cite[Theorem  16.2]{Fel25}.    
		Their second approach, outlined over $\mathbb C$, used
		$T^1$-lifting, deriving the required lifting property from
		the constancy of the ranks of certain cup-product maps
		\cite[Remark  2.6]{IM21}. 
	\end{enumerate}
	
	In the complex analytic setting, both of the above
	unobstructedness results  extend to pairs $(X,E)$ with
	$H^2(X,\operatorname{End}^{0}E)=0$, since the problem then
	reduces to that for $(X,\det E)$ \cite{IM19,Fel25};
	see Proposition  \ref{reduction}.
	Here $\operatorname{End}^{0}E$ denotes the trace-free subbundle
	of $\End E$ (see the beginning of $\S$\ref{AAAAA} for more details).

	\medskip\medskip
	The first motivation of this paper is to develop a direct
	differential-geometric approach to joint deformations based
	on the Tian--Todorov method, different from those described
	in \ref{LPapproach} and \ref{IMapproach}.
	By solving the coupled obstruction equations recursively,
	we seek non-K\"ahler extensions of joint unobstructedness,
	including cases where the Fr\"olicher spectral sequence
	does not degenerate at $E_1$.
	We also investigate the applicability of the two
	approaches of Iacono--Manetti in this setting.
	
	\medskip
The second motivation is to understand whether joint
unobstructedness persists without the trace-free
vanishing assumption.
We construct obstructed pairs in two settings:
\begin{enumerate}[{$\bullet$}]
	\item
	simple bundles on strict projective Calabi--Yau manifolds;
	\item
	Hermitian flat bundles on complex tori.
\end{enumerate}

	\subsection{Main results}
	Throughout the paper, a  \textit{Calabi--Yau manifold} means a connected compact complex manifold with torsion canonical bundle; no K\"ahler assumption is included in this terminology.  We call it  \textit{$K$-trivial} when its canonical bundle is holomorphically trivial. An $n$-dimensional  $K$-trivial Calabi--Yau manifold $X$  is called \textit{strict} if
	$H^i(X,\mathcal O_X)=0$ for every $0<i<n$.
	
	\medskip
	Our first main result concerns Fujiki manifolds.
	
	\begin{bigthm}\label{main thm}
		Let $X$ be a Calabi--Yau Fujiki manifold. 	Let $E$ be a holomorphic vector bundle of rank $r\ge 1$ over $X$ satisfying 
		$H^2(X,\End^0E)=0$. Then the
		joint deformations of  $(X,E)$ are unobstructed. 
	\end{bigthm}

	Note that	for Fujiki manifolds, the torsion assumption on $K_X$ in
	Theorem  \ref{main thm} is equivalent to Tosatti's Calabi--Yau
	condition that the first Bott--Chern class
	$c_1^{\mathrm{BC}}(X)$ vanish, or to
	$c_1(X)=0$ in $H^2(X,\mathbb R)$
	\cite[Theorem  1.5]{Tos16}.

	\medskip
	\medskip
	When the canonical bundle is trivial, we can replace the Fujiki
	assumption by \textit{weak versions} of the classical $\p\bp$-lemma. Such conditions
	have been studied by Fu--Yau \cite{FY11},
	Angella--Ugarte \cite{AU16,AU17}, Popovici--Ugarte \cite{PU18},
	Rao--Zhao \cite{RZ18}, Rao--Wan--Zhao \cite{RWZ19},
	Hu--Xia \cite{HX25} and Rao--Zhang \cite{RZ26}, among many others; see also a survey by  Alessandrini \cite{Ale19}. 
	We use the following notation from \cite{RZ18}.
	
	\medskip
	\noindent\textbf{Notation.}
	Let $0\leq p,q\leq n$. We say that an $n$-dimensional compact complex manifold $X$ satisfies $\mb S^{p,q}$  (resp. $\mb B^{p,q}$) if for every $\p \alpha \in A^{p,q}(X)$ with $\p\bp \alpha=0$,  the equation $$\bp \beta=\p \alpha$$
	admits a solution (resp. a $\partial$-exact solution).  Clearly,
	\begin{equation}\label{www}
		\p\bp\text{-lemma}\Longrightarrow \mb B^{p,q}\Longrightarrow \mb S^{p,q}, \quad\text{for every }p,q.
	\end{equation}

	\medskip
		\medskip
	The point of our second main result is that 	\textit{only three weak
	$\partial\bar\partial$-conditions suffice}.
	
	\begin{bigthm}\label{main thm2}
		Let $X$ be an $n$-dimensional $K$-trivial Calabi--Yau manifold.
		Suppose that $X$ satisfies $\mathbb S^{n,1}$,
		$\mathbb S^{n,2}$ and $\mathbb B^{n-1,2}$.
		If $E\rightarrow X$ is a holomorphic vector bundle with
		$H^2(X,\operatorname{End}^{0}E)=0$, then
		$\operatorname{Kur}(X,E)$ is smooth.
	\end{bigthm}
	
Theorem  \ref{main thm2} applies 	\textit{even when the Fr\"olicher spectral
		sequence does not degenerate at $E_1$. } We construct two such
	$K$-trivial Calabi--Yau examples in
	Propositions  \ref{prop:weak-ddbar-threefold}
	and  \ref{prop:nonzero-obstruction-example}.
	In particular, neither is a $\p\bp$-manifold, so the implication
	from the full $\p\bp$-lemma to these three weak conditions
	in \eqref{www} cannot be reversed.
	
	Here $A(E)$ denotes the Atiyah extension of $E$, and
	$H^2(X,A(E))$ is the obstruction space for joint
	deformations; see \S\ref{section 2}.
	
	\medskip
	\begin{enumerate}[{$\bullet$}]
		\item
		The first is a compact complex threefold $X$ for which
		$H^2(X,A(E))=0$ whenever $H^2(X,\End^0E)=0$.
		Thus joint unobstructedness already follows from the
		vanishing of the obstruction space.
		
		\medskip
		\item\label{second example}
		The second is a compact complex sixfold $X$ admitting a rank-two
		simple holomorphic vector bundle $E$  satisfying
		$H^2(X,\operatorname{End}^{0}E)=0$, whereas
		\begin{equation*}
			H^2(X,\operatorname{End}E)\simeq\mathbb C,
			\quad
			H^2(X,A(E))\simeq\mathbb C^2.
		\end{equation*}
		Thus our Tian--Todorov argument applies even though
		the obstruction space does not vanish.
	\end{enumerate}

	\medskip
	For a line bundle, the trace-free endomorphism bundle is zero.
	We therefore obtain:
	
	\begin{bigcor}\label{cor 1}
		Let $X$ be as in Theorem  \ref{main thm}
		or Theorem  \ref{main thm2}.
		Then $\operatorname{Kur}(X,L)$ is smooth
		for every $L\in\operatorname{Pic}(X)$.
	\end{bigcor}

	\medskip
	In dimension two, every $\p\bp$-manifold has even
	first Betti number and is therefore K\"ahler; see
	\cite[Theorem  11]{Buchdahl} or 	\cite[Theorem  IV.3.1]{BHPV04}.
	Among the three weak conditions in Theorem  \ref{main thm2},
	$\mathbb S^{2,2}$ and $\mathbb B^{1,2}$ hold automatically
	on every compact complex surface, whereas $\mathbb S^{2,1}$
	is equivalent to the K\"ahler condition;
	see Proposition  \ref{prop:weak-surface}.
	
	Recall also that $E$ is \textit{simple} if
	$$H^0(X,\End E)=\mathbb C\cdot\operatorname{id}_E;$$
	see, e.g., \cite[\S5.7]{Kob87}.
	Using Serre duality, we  obtain the following
	corollary, which removes the assumption $c_1(E)\neq0$ from
	\cite[Proposition  7.7]{CS16}.

	\begin{bigcor}\label{simple}
		Let $X$ be a K\"ahler $K$-trivial Calabi--Yau surface, and let
		$E\rightarrow X$ be a simple holomorphic vector bundle. Then  $\mathrm{Kur}(X,E)$ is smooth.
	\end{bigcor}
	
	\medskip
	\medskip
		\medskip
It is natural to ask: does the unobstructedness conclusion
of Corollary~\ref{cor 1} extend to higher-rank bundles
even when $H^2(X,\End^0E)\neq0$?
	
	 \medskip
Thomas \cite{Tho99} constructed a stable (hence simple)
rank-two bundle $E$ on a smooth hypersurface
$X\subset\mathbb P^2\times\mathbb P^2$ of bidegree $(3,3)$,
which is a strict projective Calabi--Yau threefold. 
	The deformations of $E$  with \textit{$X$ fixed} are obstructed.
	In particular, the forgetful morphism
	$\operatorname{Def}(X,E)\to\operatorname{Def}(X)$ is not smooth;
	see also \cite[Example  11.14]{Fel25}.
	Here $\operatorname{Def}(X,E)$ and $\operatorname{Def}(X)$
	denote the deformation functors of $(X,E)$ and $X$, respectively.
	This led Felten to ask the following questions.
	
	\begin{questionno}[{\cite[Remark  11.15]{Fel25}}]
		\leavevmode
		\begin{enumerate}[(i)]
			\item\label{q1}
			Is the functor $\operatorname{Def}(X,E)$
			(equivalently, the Kuranishi space $\operatorname{Kur}(X,E)$)
			smooth in Thomas's example?
			
			\item\label{q2}
			If so, does there exist a  (projective)
			Calabi--Yau manifold carrying a holomorphic vector bundle
			of rank at least two with \textit{obstructed}  joint deformations?
		\end{enumerate}
	\end{questionno}

	We answer both questions affirmatively.

	\medskip
Theorem \ref{Thomas} answers Question  \ref{q1}.\footnote{In contrast, for the non-K\"ahler sixfold in
		Proposition~\ref{prop:nonzero-obstruction-example},
		both fixed-base and joint deformations are unobstructed,
		but the forgetful morphism
		$\operatorname{Def}(X,E)\to\operatorname{Def}(X)$
		is not smooth.}
The Serre construction relates Thomas's bundle $E$ to
a rational curve $C\subset X$. In our proof, we extend
this correspondence to joint deformations, reducing
the problem to lifting $(X,C)$ across Artin extensions.
We lift the deformed curve inside
$\mathbb P^2\times\mathbb P^2$ and correct a lift of
the bicubic equation to vanish on it, without changing
the given deformation over the smaller Artin algebra.
The correspondence then recovers a lift of the bundle.

\medskip
Our next main result answers Question~\ref{q2}
\textit{even for simple bundles}.

\begin{bigthm}\label{thm:simple-obstructed-pairs}
	Let $X$ be a strict projective Calabi--Yau manifold
	of dimension $n\geq3$.
	Then there exists a simple holomorphic vector bundle
	$E\rightarrow X$ of rank at least two such that
	$\operatorname{Kur}(X,E)$ is singular.
\end{bigthm}

	Thus the higher-dimensional analogue of
	Corollary~\ref{simple} fails.
	
	\medskip

Our proof starts with a length-four subspace of $X$
whose ideal sheaf has a nonzero quadratic obstruction.
Following Joyce--Song \cite[\S8]{JS12}, we use
Seidel--Thomas twists \cite{ST01} to construct a simple
vector bundle while preserving this obstruction.
We then show that the resulting bundle extends over
every first-order deformation of $X$. This makes
the natural map from the fixed-base obstruction space
to the joint obstruction space injective, so the
obstruction persists when $X$ varies.

	\medskip
	\medskip
By contrast, Verbitsky \cite{Ver04} showed that generic
complex tori of dimension at least three admit no simple
holomorphic vector bundles of rank at least two.
Nevertheless, the following result, proved jointly with
Hisashi Kasuya in Appendix~\ref{jw}, gives obstructed pairs
with Hermitian flat bundles on every complex torus of
dimension at least two.
	\begin{bigthm}[{=Theorem  \ref{with Kasuya}}]
		Let $X$ be a complex torus of dimension $n\geq2$,
		and let $L\in\operatorname{Pic}^0(X)$. Set
		\begin{equation*}
			E:=L^{\oplus r},\quad r\geq2.
		\end{equation*}
		Then $\operatorname{Kur}(X,E)$ is singular.
	\end{bigthm}
	
Since $E$ is Hermitian flat
(i.e., it admits a Hermitian metric with vanishing
Chern curvature), the splitting theorem of
Kasuya--Purho \cite[Proposition~1.1]{KP26} gives
an isomorphism of analytic germs
$
\operatorname{Kur}(X,E)
\simeq
\operatorname{Kur}(X)\times\operatorname{Kur}_X(E),
$
where $\operatorname{Kur}_X(E)$ parametrizes deformations
of $E$ with $X$ fixed.
We identify this factor with the germ at the origin
of the variety of $n$ pairwise commuting $r\times r$
complex matrices, which is singular there.

\medskip
\medskip
\noindent{\textit{Remark.}}
While polishing this manuscript, we became aware of
Wan's preprint \cite{Wan26}.
Together with the splitting theorem of Kasuya--Purho,
his Theorem~1.5 also answers Question~\ref{q2}
affirmatively; see Remark~\ref{rem:wan-huybrechts} for more details.

	\subsection{Proof strategies for Theorems A  and  B}
We give three proofs of Theorem~\ref{main thm}, none of
which uses the Beauville--Bogomolov decomposition or
the minimal model program. The Tian--Todorov approach
also yields Theorem~\ref{main thm2}.

	\medskip
	\noindent(1) \textit{Tian--Todorov approach.}  
	By the analytic description of Huang \cite{Hua95},
	Chan--Suen \cite{CS16} and Liu--Tu \cite{LT18},
	together with Proposition  \ref{prop:formal-vs-canonical},
	it suffices to extend any first-order deformation to
	a formal solution of the Maurer--Cartan system
	\begin{equation*}
			\left\{
		\begin{aligned}
			\bar{\partial}\varphi(t)
			&= \dfrac{1}{2}\,[\varphi(t),\varphi(t)], \\[6pt]
			\bar{\partial}_{\End E} A(t)
			&=\varphi(t) \,\lrcorner\, R_E 
			+ \varphi(t) \,\lrcorner\, D^{1,0}_{\End E} A(t)
			- \dfrac{1}{2}\,[A(t),A(t)].
		\end{aligned}
		\right.
	\end{equation*}
	Here $\varphi(t)$ and $A(t)$ are formal series with values
	in $A^{0,1}(X,T_X^{1,0})$ and $A^{0,1}(X,\End E)$,
	respectively. A fixed Hermitian metric on $E$ determines
	its Chern curvature $R_E$ and the $(1,0)$-part
	$D_{\End E}^{1,0}$ of the induced connection on $\End E$.
	Unlike the first equation, the second involves both
	$\varphi(t)$ and $A(t)$.

	\medskip
	By \cite{IM19,Fel25} (see Proposition  \ref{reduction}),
	the assumption $H^2(X,\End^0E)=0$ reduces the problem
	to $(X,\det E)$. We carry out this reduction analytically
	by taking the trace of the second equation and setting
	$a(t):=\operatorname{tr}A(t)\in A^{0,1}(X)[[t]]$.

	\medskip
	First consider the torsion case. Choose $l\geq1$ with
	$lK_X\simeq\mathcal O_X$ and set $F:=(l-1)K_X$.
	Then $F$ is Hermitian flat, with flat Chern connection
	$\nabla_F=D_F^{1,0}+\bar\partial_F$.
	Fix a nowhere-vanishing section $\Omega_0$ of
	$K_X\otimes F$, viewed as an $F$-valued holomorphic
	$(n,0)$-form.
	
	\medskip
	Following an elegant idea of Liu--Rao--Wan \cite{LRW19},
	further developed in \cite[Appendix  I]{Zha26},
	we introduce an ``\textit{auxiliary}''   formal power series
	$$
	\Omega(t)=\Omega_0+\sum_{k\geq1}\Omega_k t^k,
	\quad
	\Omega_k\in A^{n,0}(X,F).
	$$
	In those works, an SNC divisor $D$ deforms locally trivially
	with $X$, and the deformation of the pair is governed by
	a single Maurer--Cartan equation for
	$\varphi(t)\in A^{0,1}(X,T_X^{1,0}(-\log D))$.
	Our construction takes $D=\varnothing$ \textit{but requires
	an additional equation for $a(t)$}, which must be solved
	\textit{simultaneously} with the complex-structure equation.

	\medskip
	The bundle-valued Tian--Todorov identity by Liu--Rao \cite{LR12}
(see Lemma  \ref{Tian-Todorov})  reduces this task
	to the following three equations:
	\begin{equation*}
		\begin{cases}
			\bigl(\bp_F + D_F^{1,0}\circ \varphi(t)\lrcorner\bigr)\Omega(t)=0,\\[0.8em]
		\bigl(\bp_F + \dfrac{1}{2}D_F^{1,0}\circ \varphi(t)\lrcorner\bigr)\bigl(\varphi(t)\lrcorner\Omega(t)\bigr)=0,\\[0.8em]
		\Bigl(		\bar\partial a(t)-\varphi(t) \lrcorner \tr R_E- \varphi(t)\lrcorner \partial a(t)\Bigr)\wedge\Omega(t)=0,\\
		\end{cases}
	\end{equation*}
	where  $\lrcorner$ denotes the contraction map. 
It is well-known that Fujiki manifolds satisfy the flat-twisted
$\p\bp$-lemma for \textit{every} Hermitian flat bundle
(see e.g., \cite[Proposition  A.1]{HZ26} for a proof).
This allows us to solve the system \textit{order by order}  and prove
Theorem  \ref{main thm}.
However, on an arbitrary $\p\bp$-manifold, the flat-twisted
$\p\bp$-lemma need not hold for every Hermitian flat bundle;
see  \cite[Theorem A]{HZ26}.

\medskip
For Theorem  \ref{main thm2}, the $K$-trivial assumption
allows us to take $F=\mathcal O_X$, and the same construction
requires only the three weak $\p\bp$-conditions
$\mathbb S^{n,1}$, $\mathbb S^{n,2}$ and
$\mathbb B^{n-1,2}$.

	\medskip
\medskip
 \noindent(2) \textit{The approaches of Iacono--Manetti.}
We also adapt both approaches of Iacono--Manetti in
\ref{IMapproach} to prove Theorem  \ref{main thm}.
For the $\mathbb P^1$-bundle argument, the key is to extend
the weak $\p\bp$-lemma for weighted logarithmic forms
in \cite[Theorem  B]{Zha26} to Fujiki manifolds.
We achieve this using K\"ahler modifications and logarithmic
Serre duality (Theorem  \ref{dbar-Fujiki}).
This yields unobstructedness of locally trivial deformations
of log Calabi--Yau pairs in the Fujiki setting
(Theorem  \ref{thm-fujiki-inobstr}).
As an application of independent interest, we also obtain
a new analytic proof of Fujino's injectivity theorem
\cite[Theorem  1.2]{Fuj17}; see Theorem  \ref{Fujino}.

\medskip
The $T^1$-lifting approach proves the $K$-trivial
$\p\bp$-manifold subcase of Theorem  \ref{main thm2}.
Passing to the canonical finite \'etale cover and applying
Wells's   pull-back injectivity theorem \cite{Wel74}
(see Lemma  \ref{wells}) then yields Theorem  \ref{main thm}.
For Theorem  \ref{main thm2} in full, the $T^1$-lifting
property follows a posteriori from our Tian--Todorov proof.
However, an independent proof through the cup-product
argument seems difficult: it remains unclear how to
establish the required lifting properties for the relevant
kernels and cokernels using only the three weak
$\p\bp$-conditions.

	\medskip
	\medskip
	This paper is organized as follows.
	\tableofcontents

	\section{Joint deformations of $(X,E)$}\label{section 2}
	In this section, following \cite{Hua95,CS16,LT18}, we review the
	deformation theory of a pair $(X,E)$, where $X$ is an
	$n$-dimensional compact complex manifold and $E\to X$ is a
	holomorphic vector bundle of rank $r$. We mainly follow the notation
	and analytic presentation of Chan--Suen \cite{CS16}. A joint deformation of $(X,E)$ can be viewed as a \textit{simultaneous deformation} of the holomorphic structures on $X$ and $E$. More precisely:
	
	\begin{definition}\label{def-joint defor}
		A \textit{family of (joint) deformations} of the pair $(X,E)$ over a  small disc $\mathbb D^m \subset \mathbb{C}^m$ centred at the origin consists of the following data:
		\begin{enumerate}[{$\centerdot$}]
			\item a proper smooth morphism $\pi: \mathcal{X} \to\mathbb D^m$
			with an isomorphism $ \pi^{-1}(0)\simeq X$;
			\item a holomorphic vector bundle $\mathcal{E} \to \mathcal{X}$ such that $\mathcal{E}|_{\pi^{-1}(0)} \simeq E$.
		\end{enumerate}
	\end{definition}
	For any $t \in\mathbb D^m$, we denote by $(X_t, E_t)$ the fibre $\pi^{-1}(t)$ together with the restriction $\mathcal{E}|_{\pi^{-1}(t)}$, respectively. Then in particular, $E_t$ is a  holomorphic vector bundle of rank $r$ over $X_t$.
	\subsection{Operators $ \overline D_t$ and $A_t$}
	Let $\left\{(X_t,E_t)\right\}_{t\in\mathbb D^m}$ be a family of deformations of the pair $(X,E)$. By Definition \ref{def-joint defor}, it in particular induces a family of deformations of the compact complex manifold $X$, which can be described by an analytic family of the so-called \textit{Beltrami differentials} $$\varphi_t\in A^{0,1}(X,T_X^{1,0})$$ satisfying the integrability condition (also called the Maurer--Cartan equation for the complex structures):
	\begin{equation*}
		\bp\varphi_t=\frac{1}{2}\,[\varphi_t,\varphi_t].
	\end{equation*}
	
	Furthermore, by Ehresmann's theorem, after possibly shrinking the disc, we may identify $\mathcal X$  with $X\times\mathbb D^m$ diffeomorphically. Let $f$ be a smooth function locally defined on $X_0$.  Then by e.g. \cite[$\S$4, Propsition 1.2]{MK71}, \cite[Proposition 2.7]{RZ18} and \cite[Proposition 3.1]{CS16}, \begin{equation}\label{Xt-hol}
		\text{$f$ is holomorphic w.r.t.  the complex structure of $X_t$} \;\Longleftrightarrow\; (\bar\partial - \varphi_t \lrcorner \partial)f = 0,
	\end{equation} 
	where $\lrcorner$ denotes the contraction operator.
	
	\medskip
	Let $\{e_i(t)\}_{1\leq i\leq r}$ be local holomorphic frames of $E_t$.  
	Thanks to  Ehresmann's theorem again, these can be regarded as local smooth frames of the fixed bundle $E$ over $X$.  
	We define an operator  
	\[
	\overline D_t : A^{0,q}(X,E) \rightarrow A^{0,q+1}(X,E)
	\]  
	by 
	\[
	 \overline D_t\left(\sum_{i=1}^r \alpha^i \otimes e_i(t)\right) := \sum_{i=1}^r \bigl( \bar\partial - \varphi_t \lrcorner \partial \bigr)\alpha^i \otimes e_i(t),
	\]
	where  
	\(\alpha^i\) are local smooth $(0,q)$-forms.

	We have the following properties thanks to the criterion \eqref{Xt-hol}:
	\begin{proposition}[{\cite[$\S$3.2]{CS16}}]
		\begin{enumerate}
			\item $ \overline D_t$ is a well-defined linear operator, i.e., it is independent of the choice of the local holomorphic frames of $E_t$. Furthermore, it satisfies the Leibniz rule:
			$$ \overline D_t (\alpha\otimes s)=\bigl (\bp -\varphi_t\l\p\bigr)\alpha\otimes s + (-1)^{|\alpha|} \alpha\wedge  \overline D_ts$$
			for any $s\in A^0(X,E)$ and $\alpha\in A^{0,\star}(X)$.  	Here $|\alpha|$ denotes the \textit{degree} of $\alpha$.
			\item For any $s\in A^0(X,E),$
			$$ \overline D_t(s)=0\Longleftrightarrow \bp_{E_t} (s)=0,$$
			where $\bp_{E_t}$ is the Dolbeault operator induced by the 
			holomorphic structure of $E_t$.
			\item $ \overline D_t^2=0.$
		\end{enumerate}
	\end{proposition}
	
	Fix an arbitrary smooth Hermitian metric $h_E$ on $E$. Let $\nabla_E$ be the Chern connection on $(E,h_E)$, which decomposes into its $(1,0)$- and $(0,1)$-parts as
	$
	\nabla_E = D^{1,0}_E + \bar\partial_E.
	$
	In the same sprit of \eqref{Xt-hol}, one naturally  to compare the operator $ \overline D_t$ and $\bp_E- \varphi_t\lrcorner  D^{1,0}_E $.
	\begin{proposition}[{\cite[Proposition 3.9]{CS16}}]
		$A_t:= \overline D_t-\bp_E+\varphi_t\lrcorner D^{1,0}_E \in A^{0,1}(X, \End E).$
	\end{proposition}
	\begin{proof}
		It is easy to show that 
		$$A_t (fs)=fA_t (s), \quad\text{for }f\in A^0(X)\,\,\text{and}\,\,s\in A^0(X,E).$$ 
	\end{proof}

	For the converse direction, assume that we are given elements $A_t\in A^{0,1}(X, \End E)$ together with $\varphi_t\in A^{0,1}(X, T_X^{1,0})$ parametrized by $t\in \mathbb D^m$. One then defines an operator $ \overline D_t: A^{0,0}(X,E)\rightarrow A^{0,1}(X,E)$ by setting
	\[
	 \overline D_t := \bar\partial_E - \varphi_t \lrcorner D_E^{1,0} + A_t.
	\]
	We also extend $ \overline D_t: A^{0,q}(X,E)\rightarrow A^{0,q+1}(X,E)$ by
	\[
	 \overline D_t (\alpha\otimes s) := \bigl(\bar\partial - \varphi_t \lrcorner \partial\bigr)\alpha \otimes s + (-1)^{|\alpha|}\, \alpha \wedge ( \overline D_t s),
	\]
	which satisfies the Leibniz rule.
	
	We then have 
	\begin{equation*}\label{Atvt}
		 \overline D_t^2 = 0 \;\Longrightarrow\; \text{the data }(A_t,\varphi_t)\text{ defines a holomorphic-pair } (X_t,E_t),
	\end{equation*}
	thanks to \cite[Theorem 3.12]{CS16}.

	\subsection{Maurer--Cartan system and DGLA for the pair}
	Let $R_E\in A^{1,1}(X,\End E)$ be the Chern curvature of $(E,h_E)$.
	The Chern connection $\nabla_E$ induces a connection on $\End E$,
	defined by
	\begin{equation*}
		(\nabla_{\End E}A)(s)
		=
		\nabla_E(As)-A(\nabla_Es).
	\end{equation*}
	We decompose
	\begin{equation}\label{EndE}
		\nabla_{\End E} \;=\; D^{1,0}_{\End E} \;+\; \bar\partial_{\End E}.
	\end{equation}
	By the Bianchi identity (see, e.g.\
	\cite[Chapter  V, Theorem  4.2]{Deme}),
	\begin{equation}\label{bianchi}
		\nabla_{\End E}R_E=0.
	\end{equation}
	In particular, $\bar\partial_{\End E}R_E=0$, and hence
	$
	[R_E]_{\bar\partial_{\End E}}
	\in
	H^{1,1}(X,\End E).
	$
	This class is independent of the choice of $h_E$ and is called the
	\textit{Atiyah class} of $E$, denoted by $\operatorname{At}(E)$.
	
	By the computations in the previous subsection, expanding $ \overline D_t^2$ leads to:
	\begin{proposition}[{\cite[p. 68, (16)-(17)]{Hua95}, \cite[\S3.3 (2)]{CS16}, \cite[Theorem 4.9]{LT18}}]\label{prop:MC-pair}
		Let $(X,E)$ be a holomorphic pair and let
		\[
		\{(A_t,\varphi_t)\}_{t\in\mathbb D^m}
		\subset A^{0,1}(X,\End E)\oplus A^{0,1}(X,T_X^{1,0})
		\]
		be a family depending holomorphically on $t$. Then $(A_t,\varphi_t)$ defines a holomorphic family of pairs $(X_t,E_t)$ if and only if it satisfies the following  Maurer--Cartan system, consisting  of two equations:
		\begin{equation}\label{MC-pair}
			\left\{
			\begin{aligned}
				\bar{\partial}\varphi_t 
				&= \dfrac{1}{2}\,[\varphi_t,\varphi_t], \\[6pt]
				\bar{\partial}_{\End E} A_t 
				&=\varphi_t \,\lrcorner\, R_E 
				+ \varphi_t \,\lrcorner\, D^{1,0}_{\End E} A_t
				- \dfrac{1}{2}\,[A_t,A_t].
			\end{aligned}
			\right.
		\end{equation}
	\end{proposition}
	
	\medskip
	Recall that (see e.g.  \cite[p. 145]{Man22}) a  \textit{differential-graded Lie algebra} (\dgla) is the data of a differential-graded vector space $(\mathrm{\mathbf L},\mathrm{\mathbf d})$ together with a bilinear map $\left[\cdot,\cdot\right]:\LLL\times\LLL\rightarrow\LLL$ (called \textit{bracket}) of degree $0$, such that the following conditions hold:
	\begin{enumerate}[{(i)}]
		\item  (graded skew-symmetry)\,\, $[a,b] = -(-1)^{|a||b|} [b,a]$;
		\item (graded Jacobi identity)\,\, $[a,[b,c]] = [[a,b],c] + (-1)^{|a||b|} [b,[a,c]]$;
		\item (graded Leibniz rule)\,\, $\mathrm{\mathbf{d}}[a,b] = [\mathrm{\mathbf{d}} a, b] + (-1)^{|a|} [a, \mathrm{\mathbf{d}} b]$.
	\end{enumerate}

	\medskip
	By \cite[$\S$1.1]{Hua95} and also \cite[$\S$3]{CS16},
	set
	\[
	A(E) := \End E\oplus T_X^{1,0}
	\]
	as a smooth complex vector bundle over $X$. For $(A,\varphi)\in A^{0,p}(X,A(E))$ and $(B,\psi)\in A^{0,q}(X,A(E))$, define a bracket
	\[
	[\cdot,\cdot] : A^{0,p}(X,A(E))\times A^{0,q}(X,A(E)) \longrightarrow A^{0,p+q}(X,A(E))
	\]
	by
	\[
	[(A,\varphi),(B,\psi)]
	\;:=\;
	\bigl(\,\varphi\lrcorner\,D^{1,0}_{\End E} B
	\;-\;(-1)^{pq}\,\psi\lrcorner\,D^{1,0}_{\End E} A
	\;-\;[A,B]\,,\;[\varphi,\psi]\,\bigr).
	\]
	
	Define
	\[
	\bar\partial_{A(E)}^{\Im} : A^{0,p}(X,A(E)) \longrightarrow A^{0,p+1}(X,A(E))
	\]
	on $A(E)$ by
	\[
	\bar\partial_{A(E)}^{\Im} :=
	\begin{pmatrix}
		\bar\partial_{\End E} & \Im \\
		0 & \bar\partial_{T_X}
	\end{pmatrix},
	\]
	where $\Im\in A^{0,1}\bigl(X,\mathrm{Hom}(T_X^{1,0},\End E)\bigr)$ is defined by
	\[
	\Im\wedge \varphi := (-1)^p\,\varphi\lrcorner R_E,
	\quad \varphi\in A^{0,p}(X,T_X^{1,0}).
	\]

	One checks that $\bigl(A(E),\bar\partial_{A(E)}^{\Im}\bigr)$ defines a holomorphic vector bundle over $X$, and that its holomorphic structure depends only on the Dolbeault cohomology class $[\Im]$, hence only on $\mathrm{At}(E)$,  the Atiyah class of $E$. Thus we will simply write $\bar\partial_{A(E)}$ and refer to $(A(E),\bp_{A(E)})$ as the \textit{Atiyah extension} of $E$.
	
	\medskip
	We can verify that $$(A^{0,\bullet}(X,A(E)),\bar\partial_{A(E)},[\cdot,\cdot])$$ is a DGLA. 
	Moreover, the Maurer--Cartan system \eqref{MC-pair} can be packaged into a \textit{single Maurer--Cartan equation in $A(E)$}:
	\begin{equation}\label{one-MC}
		\bar\partial_{A(E)}(A_t,\varphi_t)
		=
		\frac{1}{2}\,[(A_t,\varphi_t),(A_t,\varphi_t)].
	\end{equation}
	
	\begin{remark}\label{rmk-iso}
		In the algebraic setting, the role of the Atiyah extension $A(E)$ is played by the sheaf $\mathcal D^1(E)$ of scalar differential operators of order at most one on $E$ (see \cite{Ser06} for the line bundle case and \cite{Mar09} for the general case). Recall that a \textit{differential operator of order $1$} on a holomorphic vector bundle $E$ of rank $r$ is a linear map $P\colon \Omega^0(E)\to \Omega^0(E)$ which, locally, takes the form
		\[
		P = (g_{ij}) + \sum_{k=1}^n (h_{ij}^k) \frac{\partial}{\partial z^k},
		\]
		where $(g_{ij})$ and $(h_{ij}^k)$ are $r\times r$ matrices of holomorphic functions. Such an operator is said to have \textit{scalar principal symbol} if $(h_{ij}^k) = h^k \cdot \mathrm I_{r\times r}$ for some local holomorphic functions $h^k$. There exists an exact sequence
		\[
		0 \longrightarrow \operatorname{End}E \longrightarrow \mathcal D^1(E) \xrightarrow[\quad\,]{\sigma} T_X^{1,0} \longrightarrow 0,
		\]
		where the symbol map $\sigma$ is locally given by $\sigma(P)=\sum_{k=1}^n h^k \frac{\partial}{\partial z_k}$. It was shown in e.g. \cite[Proposition A.2]{CS16} that $\mathcal D^1(E)$ carries a natural holomorphic structure, under which it is isomorphic to the Atiyah extension $A(E)$. Therefore, the DGLA $(A^{0,\bullet}(X,A(E)), \bar\partial_{A(E)}, [\cdot,\cdot])$ above is naturally isomorphic to the one obtained via the algebraic construction.
		
	\end{remark}

	\subsection{Kuranishi family}
	For a holomorphic pair $(X,E)$, the existence of the corresponding Kuranishi family was proved in \cite[Theorem 1.8]{Hua95} and \cite[\S\S5-6]{CS16}. We now follow the analytic construction of Chan--Suen in this subsection.

	Choose any smooth Hermitian metrics $\omega$ on $X$ and $g$ on $A(E)$.  
	By Hodge theory on $A^{0,\bullet}(X,A(E))$, let $\mathcal H$ (resp.\ $\mathcal G$, $\bar\partial_{A(E)}^*$) denote the harmonic projection (resp.\ the Green operator, the formal adjoint of $\bar\partial_{A(E)}$) with respect to $\omega$ and $g$.

	Let $\{\eta_1,\dots,\eta_m\}$ be a basis of $H^1(X,A(E))\simeq \mathcal H^1$, and set
	\begin{equation}\label{e1}
		\epsilon_1(t) := \sum_{i=1}^m t_i\,\eta_i \in \mathcal{H}^1.
	\end{equation}
	Consider the non-linear equation
	\begin{equation}\label{eq:Kuranishi-fixed}
		\epsilon(t) \;=\; \epsilon_1(t) + \frac12\,
		\bar\partial_{A(E)}^{*}\,\mathcal{G}\bigl([\epsilon(t),\epsilon(t)]\bigr).
	\end{equation}
	Then there exists a unique canonical formal power series  solution of \eqref{eq:Kuranishi-fixed},
	\begin{equation}\label{good}
		\epsilon^{\mathrm{can}}(t)
		=
		\sum_{j\ge 1}\epsilon^{\mathrm{can}}_j(t)
		\,\in\, A^{0,1}(X,A(E)),
	\end{equation}
	where each $\epsilon^{\mathrm{can}}_j(t)$ is homogeneous of degree $j$ in $t.$ Indeed, this is clear since 
	\begin{equation}\label{epsilon2}
		\begin{aligned}
			\epsilon^{\mathrm{can}}_2(t) &= \frac{1}{2} \bp^*_{A(E)}\mathcal{G} \bigl[\epsilon_1(t), \epsilon_1(t)\bigr], \\
			\epsilon^{\mathrm{can}}_3(t) &= \frac{1}{2} \bp^*_{A(E)}\mathcal{G} \left([\epsilon_1(t),\epsilon^{\mathrm{can}}_2(t)]+[\epsilon^{\mathrm{can}}_2(t),\epsilon_1(t)]\right), \\
			&\vdots \\
			\epsilon^{\mathrm{can}}_N(t) &= \frac{1}{2} \bp^*_{A(E)} \mathcal{G} \Biggl( \bigl[\epsilon_1(t),\epsilon^{\mathrm{can}}_{N-1}(t)\bigr] + \bigl[\epsilon^{\mathrm{can}}_{N-1}(t),\epsilon_1(t)\bigr] \\
			&\qquad\qquad\qquad + \sum_{k=2}^{N-2} \bigl[\epsilon^{\mathrm{can}}_k(t), \epsilon^{\mathrm{can}}_{N-k}(t)\bigr] \Biggr), \quad N\ge 4.
		\end{aligned}
	\end{equation}
	Another important property is that, by standard elliptic regularity, 
	$\epsilon^{\mathrm{can}}(t)$ depends analytically on $t$ for \textit{sufficiently small $|t|$}.
	Moreover,
	\begin{equation*}\label{eq:MC-equiv}
		\begin{aligned}
			&\epsilon^{\mathrm{can}}(t) \text{ satisfies the Maurer--Cartan equation }\\
			&\qquad\quad\;\;\;\;\bar\partial_{A(E)}\epsilon^{\mathrm{can}}(t) = \frac12\,[\epsilon^{\mathrm{can}}(t),\epsilon^{\mathrm{can}}(t)]\\[4pt]
			&\qquad\qquad\Longleftrightarrow\quad 
			\mathcal{H}\bigl([\epsilon^{\mathrm{can}}(t),\epsilon^{\mathrm{can}}(t)]\bigr) = 0 .
		\end{aligned}
	\end{equation*}
	
	The associated \textit{Kuranishi obstruction map} is
	\[
	\textrm{Ob}_{(X,E)}(t) : U \subset H^1(X,A(E)) \longrightarrow H^2(X,A(E)),\quad
	t \longmapsto \mathcal H\bigl([\epsilon^{\mathrm{can}}(t),\epsilon^{\mathrm{can}}(t)]\bigr),
	\]
	defined on a suitable small neighborhood $U$ of $0\in H^1(X,A(E))$. 
	The \emph{Kuranishi space} of the pair \((X,E)\) is defined to be the analytic subspace
	\[
	\operatorname{Kur}(X,E)
	:=\textrm{Ob}^{-1}_{(X,E)} (0)=
	\left\{
	t\in  U
	\ \middle|\
	\mathcal H\bigl(
	[\epsilon^{\mathrm{can}}(t),\epsilon^{\mathrm{can}}(t)]
	\bigr)=0
	\right\}
	\]
	
	For each \(s\in\operatorname{Kur}(X,E)\), the element
	\(
	\epsilon^{\mathrm{can}}(s)
	=
	\bigl(A_s^{\mathrm{can}},\varphi_s^{\mathrm{can}}\bigr)
	\)
	satisfies the Maurer--Cartan system \eqref{MC-pair} and hence, by Proposition \ref{prop:MC-pair}, determines a holomorphic pair \((X_s,E_s)\). As \(s\) varies, these pairs form the \textit{Kuranishi family}
	\(
	\left\{(X_s,E_s)\right\}_{s\in\operatorname{Kur}(X,E)}.
	\)
	
	\begin{remark}
		The above construction involves the choice of Hermitian metrics, and hence the resulting analytic space \(\operatorname{Kur}(X,E)\) is \textit{not uniquely defined}. Nevertheless, the germs at the origin arising from different choices of Hermitian metrics are (non-canonically) isomorphic. We denote this 
		\textit{Kuranishi germ} by
		$\bigl(\operatorname{Kur}(X,E),0\bigr)$.
	\end{remark}
	
	\begin{definition}
		The pair \((X,E)\) is said to have \emph{unobstructed joint deformations} if its Kuranishi germ at the origin is smooth.
	\end{definition}
	
	\begin{proposition}\label{prop:formal-vs-canonical}
		The joint deformations of   $(X,E)$ are unobstructed  if and only if for $\epsilon_1(t)$ defined in \eqref{e1},
		the   obstruction equation
		\[
		\bar\partial_{A(E)}\epsilon(t)
		=
		\frac12\,[\epsilon(t),\epsilon(t)],
		\]
		or equivalently, writing $\epsilon(t) = (A_t,\varphi_t)$, 
		\begin{empheq}[left=\empheqlbrace]{align}
			\bar{\partial}\varphi_t &= \frac{1}{2}[\varphi_t,\varphi_t];  \tag{MC1}\label{MC1} \\[6pt]
			\bar{\partial}_{\End E} A_t &= \varphi_t \lrcorner R_E + \varphi_t \lrcorner D_{\End E}^{1,0} A_t - \frac{1}{2}[A_t,A_t] \tag{MC2}\label{MC2},
		\end{empheq}
		can be solved order-by-order formally.
	\end{proposition}
	
	\begin{proof}
		The “only if’’ part is clear by setting $\epsilon(t):=\epsilon^{\mathrm{can}}(t)$ as in \eqref{good}.
		
		For the converse, the degree-two  part of the Maurer--Cartan equation reads
		\[
		\bar\partial_{A(E)}\epsilon_2(t)
		=
		\frac{1}{2}\,[\epsilon_1(t),\epsilon_1(t)].
		\]
		In particular, $[\epsilon_1(t),\epsilon_1(t)]$ is $\bar\partial_{A(E)}$-exact. By bundle-valued Hodge decomposition,
		\[
		[\epsilon_1(t),\epsilon_1(t)]
		=
		\mathcal H[\epsilon_1(t),\epsilon_1(t)]
		+
		\bar\partial_{A(E)}\bar\partial_{A(E)}^*\mathcal G[\epsilon_1(t),\epsilon_1(t)]
		+
		\bar\partial_{A(E)}^*\bar\partial_{A(E)}\mathcal G[\epsilon_1(t),\epsilon_1(t)],
		\]
		and exactness forces $\mathcal H[\epsilon_1(t),\epsilon_1(t)]=0$, hence
		\[
		[\epsilon_1(t),\epsilon_1(t)]
		=
		\bar\partial_{A(E)}\bar\partial_{A(E)}^*\mathcal G[\epsilon_1(t),\epsilon_1(t)].
		\]
		We then replace $\epsilon_2(t)$ by the degree-two term of the canonical solution of \eqref{good}, namely
		\[
		\epsilon_2^{\mathrm{can}}(t)
		=
		\frac{1}{2}\,\bar\partial_{A(E)}^*\mathcal G[\epsilon_1(t),\epsilon_1(t)]
		\]
		presented in \eqref{epsilon2}, which  still solves the  Maurer--Cartan equation in degree-two.
		
		\medskip
	At each higher order $N \ge 3$, a slightly more involved
	argument shows that one may replace $\epsilon_N(t)$ by
	the ``better'' term $\epsilon_N^{\mathrm{can}}(t)$, so that
	the full (convergent) power series $\epsilon^{\mathrm{can}}(t)$
	as in \eqref{good} still satisfies the Maurer--Cartan equation
	over a sufficiently small disc. This yields the desired
	unobstructedness.
	\end{proof}

	\section{Theorems  \ref{main thm},   \ref{main thm2}  and Corollary \ref{simple}: proofs and related examples}\label{AAAAA}
	
	\subsection{Tian--Todorov proofs of Theorem \ref{main thm}  and  Theorem \ref{main thm2}}
	For simplicity, we will always assume $$m =  \dim_{\mathbb{C}} H^1(X,A(E))=1$$ in what follows, so that there is only one deformation parameter $t:=t_1$. The case \( m > 1 \) can be treated in a similar way.

	\medskip
	\medskip
	Let $E$ be a holomorphic vector bundle of rank $r$ over $X$. The trace map in linear algebra gives rise to an $\mathcal O_X$-linear \textit{trace map}
	\[
	\operatorname{tr}\colon \End E \longrightarrow \mathcal{O}_X,
	\]
	satisfying $\operatorname{tr}(\operatorname{id}_E)=r$. Denote the kernel of the trace map by $\End^0 E$, the \textit{trace-free part} of the endomorphism bundle. The trace map splits \textit{holomorphically} via the scaled diagonal embedding
	\[
	\mathcal{O}_X \longrightarrow \End E,\quad \lambda \longmapsto \frac{\lambda}{r}\,\operatorname{id}_E,
	\]
	so that
	\[
	\End E = \End^0 E \oplus \mathcal{O}_X
	\]
	as holomorphic vector bundles. Consequently,
	\begin{equation}\label{direct sum}
		H^k(X,\End E) = H^k(X,\End^0 E) \oplus H^k(X,\mathcal{O}_X)
	\end{equation}
	for any $k$.
	
	\medskip
	\medskip
	Fix a Hermitian metric $h_E$ on $E$ and let $\nabla_E = D_E^{1,0} + \bar\partial_E$ be its associated Chern connection. Denote by $R_E$ the curvature of $\nabla_E$.
	In the sequel we will always adopt the following notation. For any $A \in A^{0,1}(X,\End E)$, write
	\[
	A = A^0 + \frac{a}{r}\,\operatorname{id}_E,\quad \operatorname{tr} A^0 = 0,\quad a = \operatorname{tr} A \in A^{0,1}(X,\mathcal O_X).
	\]
	Correspondingly, split the curvature:
	\[
	R_E = R_E^0 + \frac{1}{r}(\operatorname{tr} R_E)\,\operatorname{id}_E,\quad \operatorname{tr} R_E^0 = 0.
	\]
	Here $\operatorname{tr} R_E$ is the curvature of $\det E$, which  is  $(1,1)$-form  satisfying
	\begin{equation}\label{bianchi2}
		d \tr R_E=0
	\end{equation}by the Bianchi identity \eqref{bianchi}.

	\medskip\medskip
	\begin{proof}[{Proof of Theorem \ref{main thm}}]
		The proof  will be divided into two steps.
		\begin{step}{Reduction on the obstruction equations:   Calabi--Yau property \& assumption $H^2(X,\End ^0 E)=0$.}\label{step1}
		\end{step}
		
		Under the assumption that $K_X$ is torsion, i.e., $lK_X$ is holomorphically trivial for some integer $l\geq  1$, write
		\begin{equation} \label{bundle F}
			lK_X = K_X + (l-1)K_X \quad \text{and set} \quad F := (l-1)K_X.
		\end{equation}
		Since $F$ is also torsion, it is  Hermitian  flat. More precisely, we can equip $F$ with a Hermitian metric $h_F$ whose Chern connection $$\nabla_F = D_F^{1,0} + \bar\partial_F$$ is flat, meaning its curvature $R_F$ vanishes.
		
		Moreover, there exists a nowhere vanishing $F$-valued holomorphic $(n,0)$-form
		\[
		\Upomega \in H^0(X, lK_X) \simeq H^{n,0}(X, F).
		\]

		We will need the following bundle-valued version of the Tian--Todorov lemma, in which the flatness of $F$ is not required --- it will be needed  later; see Lemma \ref{Fujiki}.
		
		\begin{lemma}[{\cite[Corollary 3.4]{LR12}}]\label{Tian-Todorov}
			For any $\varphi,\psi\in A^{0,1}\bigl(X, T_X^{1,0}\bigr)$ and any $\eta\in A^{\bullet,\bullet}(X,F)$, one has
			\[
			[\varphi,\psi] \lrcorner \eta
			= -D_F^{1,0}\bigl(\psi \lrcorner \bigl(\varphi \lrcorner \eta\bigr)\bigr)
			+ \psi \lrcorner D_F^{1,0}\bigl(\varphi \lrcorner \eta\bigr)
			+ \varphi \lrcorner D_F^{1,0}\bigl(\psi \lrcorner \eta\bigr)-\psi \l\bigl(\varphi\l D_F^{1,0} \eta\bigr)
			\]
		\end{lemma}

		Using the Calabi--Yau property of $X$ and following the insightful idea dating back to \cite{LRW19}  (also developed in \cite{Zha26}), we will need another auxiliary formal power series $\Omega(t)$ with $\Omega(0)=\Upomega$ to further reduce Proposition \ref{prop:formal-vs-canonical}   to the following:
		
		\begin{proposition}\label{prop-1}
			In the setting of Theorem  \ref{main thm},     for any first-order deformation 
			$\epsilon_1=(A_1,\varphi_1)$, suppose that there exist formal power series
			\[
			\varphi(t)=\sum_{i\geq 1}\varphi_i t^i \in A^{0,1}(X, T_X^{1,0})[[t]],\quad
			a(t)=\sum_{i\geq 1} a_i t^i  \in A^{0,1}(X,\mathcal O_X)[[t]], \quad a_1=\operatorname{tr} A_1,
			\]
			and
			\[
			\Omega(t)=\Omega_0+\sum_{i\geq 1}\Omega_i t^i\in A^{n,0}(X, F)[[t]],\quad	\Omega_0=\Upomega,
			\]
			satisfying the system formally:
			\begin{equation}\tag{$\dagger$}\label{log 0.2}
				\begin{cases}
					\bigl(\bp_F + D_F^{1,0}\circ \varphi(t)\lrcorner\bigr)\Omega(t)=0,\\[0.8em]
					\bigl(\bp_F + \dfrac{1}{2}D_F^{1,0}\circ \varphi(t)\lrcorner\bigr)\bigl(\varphi(t)\lrcorner\Omega(t)\bigr)=0,\\[0.8em]
					\Bigl(		\bar\partial a(t)-\varphi(t) \lrcorner \tr R_E- \varphi(t)\lrcorner \partial a(t)\Bigr)\wedge\Omega(t)=0.\\
				\end{cases}
			\end{equation}
			Then the Kuranishi space of the pair $(X,E)$ is smooth.
		\end{proposition}
		
		\begin{proof}
			The first two 	equations in  \eqref{log 0.2} together with Lemma \ref{Tian-Todorov} (noting that $\Omega(t)$ is automatically $D_F^{1,0}$-closed for type reasons) imply
			\begin{align*}
				\bp\bigl(\varphi(t)\lrcorner \Omega(t)\bigr)
				&= -\frac{1}{2}\,D^{1,0}_F\Bigl(\varphi(t)\lrcorner\bigl(\varphi(t)\lrcorner\Omega(t)\bigr)\Bigr)\\
				&= \frac{1}{2}\,[\varphi(t),\varphi(t)]\lrcorner \Omega(t)
				- \varphi(t)\lrcorner D^{1,0}_F\bigl(\varphi(t)\lrcorner\Omega(t)\bigr)\\
				&= \frac{1}{2}\,[\varphi(t),\varphi(t)]\lrcorner \Omega(t)
				+ \varphi(t)\lrcorner\bp\Omega(t).
			\end{align*}
			Hence
			\begin{equation}\label{aabb}
				\bigl(\bp\varphi(t)\bigr)\lrcorner \Omega(t)
				= \bp\bigl(\varphi(t)\lrcorner \Omega(t)\bigr)
				- \varphi(t)\lrcorner\bp\Omega(t)
				= \frac{1}{2}\,[\varphi(t),\varphi(t)]\lrcorner \Omega(t).
			\end{equation}
			
			Set 
			\[
			V(t) := \bp\varphi(t) - \frac{1}{2}[\varphi(t),\varphi(t)] = \sum_{i\ge 2} v_i t^i,
			\]
			with \(v_i \in A^{0,2}(X,T_X^{1,0})\). Then \eqref{aabb} becomes \(V(t)\lrcorner\Omega(t)=0\).
			
			Locally write 
			\[
			\Upomega = f(z)\,dz^1\wedge\cdots\wedge dz^n\otimes e
			\]
			with \(f(z)\) nowhere vanishing and \(e\) a local frame of \(F\). Any \(v\in A^{0,2}(X,T_X^{1,0})\) can be expressed locally as
			\[
			v = \sum_{j=1}^n g_j\,\frac{\partial}{\partial z^j},
			\quad\text{where } g_j \text{ are smooth }(0,2)\text{-forms}.
			\]
			Contracting with \(\Upomega\) gives
			\[
			\Upomega\lrcorner v = f(z)\sum_{j=1}^n (-1)^{j-1} g_j\,
			dz^1\wedge\cdots\wedge\widehat{dz^j}\wedge\cdots\wedge dz^n\otimes e.
			\]
			The \((n-1,2)\)-forms appearing are linearly independent, hence the map \(v\mapsto\Upomega\lrcorner v\) is injective.

			\medskip
			Expanding
			\[
			V(t)\lrcorner\Omega(t)=\sum_{k\ge 2}\Bigl(\sum_{i+j=k}v_i\lrcorner\Omega_j\Bigr)t^k=0,
			\]
			we obtain:
			
			\begin{itemize}	\setlength{\itemsep}{5pt}
				\item[$\centerdot$] For \(k=2\), \(v_2\lrcorner\Upomega=0\), so \(v_2=0\) by injectivity.
				\item[$\centerdot$] Assume \(v_2=\dots=v_{N}=0\) for some \(N\ge 2\). Then the coefficient of \(t^{N+1}\) reduces to \(v_{N+1}\lrcorner\Upomega=0\), hence \(v_{N+1}=0\).
			\end{itemize}
			Therefore  \(V(t)=0\), i.e., 
			\begin{equation}\label{hhh}
				\bp\varphi(t) = \dfrac{1}{2}[\varphi(t),\varphi(t)],
			\end{equation}
			namely the first obstruction equation \eqref{MC1}.

			\medskip
			\medskip
			Via a similar argument, the third equation in \eqref{log 0.2} is equivalent to 
			\begin{equation}
				\bar\partial a(t)=\varphi(t) \lrcorner \tr R_E+ \varphi(t)\lrcorner \partial a(t).
			\end{equation}
			This is obtained by taking the trace of the second obstruction equation \eqref{MC2} wherein the following identities are used:
			
			\smallskip
			\begin{itemize}
				\setlength{\itemsep}{7.9pt}
				\item[--\quad] $\tr(\bp_{\End E}A(t)) = \bp(\tr A(t))$;
				\item[--\quad] $\tr(\varphi(t)\lrcorner R_E) = \varphi(t)\lrcorner\tr R_E$;
				\item[--\quad]$\tr([A(t),A(t)]) = 0$;
				\item[--\quad] For the term $\tr\Big(\varphi(t)\lrcorner D_{\End E}^{1,0}A(t)\Big)$, we compute locally. In a local trivialisation, $$D_{\End E}^{1,0}A(t) = \partial A(t) + [\omega^{1,0},A(t)],$$ where $\omega^{1,0}$ is a matrix-valued $(1,0)$-form. Then this term equals
				\[
				\varphi(t)\lrcorner\tr\Big(D_{\End E}^{1,0}A(t)\Big) = \varphi(t)\lrcorner\Big(\tr(\partial A(t)) + \tr([\omega^{1,0},A(t)])\Big) = \varphi(t)\lrcorner\partial a(t).
				\]
			\end{itemize}
			
			\medskip
			Consequently, the trace-free part of \eqref{MC2} becomes
			\begin{equation}\tag{TF}\label{tf}
				\bar\partial_{\End E} A^0(t) = \varphi(t) \lrcorner R_E^0 + \varphi(t) \lrcorner D_{\End E}^{1,0} A^0(t) - \frac{1}{2}[A(t),A(t)].
			\end{equation}
			
			Expanding in $t$, the $N$-th order obstruction equation for $N \ge 2$ takes the form
			\[
			\bar\partial A_N^0 = C_N^{\text{tf}},
			\]
			where $C_N^{\text{tf}}$ is $\bar\partial_{\End E}$-closed, \textit{provided that all obstructions up to order $N-1$ in \eqref{tf} and all obstructions up to order $N$ in \eqref{hhh} vanish}.  Since $H^2(X,\End^0 E)=0$, we can solve for $A_N^0$ inductively. Hence
			\[
			A(t): = A^0(t) + \frac{a(t)}{r}\,\operatorname{id}_E
			\]
			satisfies \eqref{MC2}. This completes the proof. 
		\end{proof}
		
		\begin{remark}
			The main differences between Proposition \ref{prop-1} and \cite[Proposition 4.10]{LRW19}, as well as \cite[Proposition I.3]{Zha26} (both with $D=\varnothing$), are two-fold:
			\begin{enumerate}[(i)]
				\item An additional obstruction equation (the third equation in \eqref{log 0.2}) needs to be solved simultaneously;
				\item The solutions $\varphi(t)$, $a(t)$ and $\Omega(t)$ are only required to be formal. 
				\begin{enumerate}[label=(\alph*)]
					\item In the arguments of \cite{LRW19,Zha26}, the smoothness of the corresponding  $\Omega(t)$  was used to ensure that $\Omega(t)$ has no zeros for small $t$ (since $\Omega(0)$ is nowhere vanishing).  However, formal solutions suffice as shown above.
					\item   As for $\varphi(t)$ and $a(t)$, formal solutions are also enough: once they are obtained, we can then choose particularly nice solutions whose smoothness is very standard (see Proposition \ref{prop:formal-vs-canonical}). 
				\end{enumerate}
				This omits the regularity arguments  as in \cite[proof of Theorem 4.12, pp. 36-39]{LRW19} and \cite[Appendix I, Step (III)]{Zha26}.
			\end{enumerate}
		\end{remark}
		
		\begin{step}{Solve the obstruction equations: iteration}\label{step2}
		\end{step}
		
		By \eqref{log 0.2}$_k$ we mean the part of \eqref{log 0.2} formed by the first equation at order $k$ together with the second and third equations at order $k+1$. We now adopt an \textit{iterative} approach to solve \eqref{log 0.2}.

		Before proceeding with the proof, we recall the following standard
		fact; see, for example, \cite[Proposition A.1]{HZ26} for a proof.
		
		\begin{lemma}\label{Fujiki}
			Let $X$ be a Fujiki manifold, and let $(F,h_F)$ be a Hermitian flat holomorphic
			vector bundle on $X$. Write its Chern connection as
			$
			\nabla_F=D_F^{1,0}+\bar\partial_F.
			$
			Then $(X,F)$ satisfies the
			$D_F^{1,0}\bar\partial_F$-lemma. More precisely, for every
			$\nabla_F$-closed $F$-valued $(p,q)$-form $\alpha$ on $X$, the
			following conditions are equivalent:
			\begin{equation*}
				\nabla_F\text{-exactness}
				\Longleftrightarrow
				D_F^{1,0}\text{-exactness}
				\Longleftrightarrow
				\bar\partial_F\text{-exactness}
				\Longleftrightarrow
				D_F^{1,0}\bar\partial_F\text{-exactness}.
			\end{equation*}
		\end{lemma}
		
		For $k=1$, the system \eqref{log 0.2}$_1$ reads as
		\begin{equation}\label{123}
			\begin{dcases}
				\bar{\partial}_F\Omega_1 + D_F^{1,0} \bigl(\varphi_1 \lrcorner \Omega_0\bigr) = 0, \\[0.6em]
				\bar{\partial}_F\bigl(\varphi_2 \lrcorner \Omega_0 + \varphi_1 \lrcorner \Omega_1\bigr)
				+ \frac{1}{2}\, D_F^{1,0} \bigl(\varphi_1 \lrcorner (\varphi_1 \lrcorner \Omega_0)\bigr) = 0, \\[0.6em]
				\begin{aligned}
					\bar{\partial}a_2 \wedge \Omega_0 
					= &-\bar\partial a_1\wedge\Omega _1 + \bigl(\varphi_2 \lrcorner\tr  R_E \bigr) \wedge \Omega_0  +  \bigl(\varphi_1 \lrcorner \tr R_E \bigr) \wedge \Omega_1 +  \bigl(\varphi_1 \l \p a_1 \bigr) \wedge \Omega_0. 
				\end{aligned}
			\end{dcases}
		\end{equation}
		From the first equation in \eqref{123}, we have
		\[
		\bar{\partial}_F D_F^{1,0}  (\varphi_1 \lrcorner \Omega_0)
		= - D_F^{1,0} \bar{\partial}_F (\varphi_1 \lrcorner \Omega_0)
		= - D_F^{1,0}  \bigl( (\bar{\partial} \varphi_1) \lrcorner \Omega_0 + \varphi_1 \lrcorner \bar{\partial}_F \Omega_0 \bigr) = 0,
		\]
		since $\bar{\partial}_F \Omega_0 = 0$ and $\bar{\partial} \varphi_1 = 0$. Hence, there exists 
		\begin{equation}\label{exact1}
			\textrm{a solution } \Omega_1
		\end{equation}
		thanks to Lemma \ref{Fujiki}.

		\medskip
		Likewise in the second equation of \eqref{123}, since $D_F^{1,0}\bigl( \varphi_1 \lrcorner (\varphi_1 \lrcorner \Omega_0) \bigr)$ is $\bar{\partial}_F$-closed,  Lemma \ref{Fujiki} can be applied  again to $$\textrm{solve for } \varphi_2 \lrcorner \Omega_0,   \textrm{ and thus for } \varphi_2 \textrm{ itself,}$$ using the non-degeneracy of $\Omega_0$. Furthermore, one can ensure the existence of some $B_2\in A^{n-2,1}(X,F)$ such that 
		\begin{equation}\label{exact2}
			\varphi_2 \lrcorner \Omega_0 + \varphi_1 \lrcorner \Omega_1 = D_F^{1,0} B_2.
		\end{equation}

		\medskip
		For the third equation in \eqref{123}, its right-hand side is $\bar{\partial}_{ F}$-closed, since we have already solved for $\varphi_2$.
		For any $\theta \in A^{0,1}(X, T_X^{1,0})$ and bundle-valued forms $u, v$, we have \cite[p. 292]{Deme}
		\begin{equation}\label{contraction-formula}
			\theta \lrcorner (u \wedge v) = (\theta \lrcorner u) \wedge v + (-1)^{\deg u} u \wedge (\theta \lrcorner v).
		\end{equation}
		Therefore, by type reasons,
		\begin{equation}\label{eq:first}
			\bar{\partial} a_2 \wedge \Omega_0
			= -\bar\partial a_1 \wedge \Omega_1
			- \operatorname{tr} R_E \wedge (\varphi_2 \l \Omega_0 + \varphi_1 \lrcorner \Omega_1)
			- \partial a_1 \wedge (\varphi_1 \lrcorner \Omega_0).
		\end{equation}
		
		Using \eqref{exact2}, \eqref{bianchi2}, \eqref{exact1}, and the Leibniz rule, we obtain
		\begin{equation}\label{eq:second}
			\begin{aligned}
				\bar{\partial}_F \bigl(a_2 \wedge \Omega_0\bigr)
				&= -\bar\partial a_1 \wedge \Omega_1
				- \operatorname{tr} R_E \wedge (D_F^{1,0} B_2)
				- \partial a_1 \wedge (\varphi_1 \lrcorner \Omega_0) \\[0.5em]
				&= -\bar\partial a_1 \wedge \Omega_1
				- D_F^{1,0}\bigl( \operatorname{tr} R_E \wedge B_2 - a_1 \wedge (\varphi_1 \lrcorner \Omega_0) \bigr)
				+ a_1 \wedge \bar\partial_F \Omega_1 \\[0.5em]
				&= -\bar\partial_F (a_1 \wedge \Omega_1)
				- D_F^{1,0}\bigl( \operatorname{tr} R_E \wedge B_2 - a_1 \wedge (\varphi_1 \lrcorner \Omega_0) \bigr).
			\end{aligned}
		\end{equation}
		Therefore, we can \begin{equation}\label{exact3}
			\textrm{solve for $a_2 \wedge \Omega_0 + a_1 \wedge \Omega_1$, and hence for $a_2$},
		\end{equation}  based on Lemma \ref{Fujiki}.

		\medskip
		\medskip
		\medskip
		By induction, assume that \eqref{log 0.2}  is solvable for all $k \le N$, and that we have already found
		\[
		\Omega_i, \, 0 \le i \le N, \quad \varphi_i, \, 1 \le i \le N+1,  \quad \text{and} \quad a_i,  \, 1 \le i \le N+1.
		\]
		Thus, \eqref{hhh} implies
		\begin{equation}\label{log 0.4}
			\bar{\partial} \varphi_k = \frac{1}{2} \sum_{i+j=k} [\varphi_i, \varphi_j], \quad k \le N+1,
		\end{equation}
		and then 
		\begin{equation}\label{log 0.5}
			\bar{\partial} \Bigl( \sum_{i+j=k} [\varphi_i, \varphi_j] \Bigr) = 0, \quad k \le N+2,
		\end{equation}
		by using $[\alpha,\beta]=-[\beta,\alpha]$ for $\alpha\in A^{0,2} (X,T_X^{1,0})$ and arbitrary $\beta$.

		\medskip
		\medskip
		For the $(N+1)$-th step, \eqref{log 0.2}$_{N+1}$ together with \eqref{contraction-formula} yields
		\begin{equation}\label{log 0.6}
			\begin{dcases}
				\bar{\partial}_F \Omega_{N+1} 
				+ \sum_{\substack{i+j=N+1}} D^{1,0}_{F}  \bigl({\varphi_i}\lrcorner \Omega_j \bigr) = 0, \\[0.8em]
				\sum_{\substack{i+j=N+2}} \bar{\partial}_F \bigl( {\varphi_i}\lrcorner \Omega_j \bigr)
				+ \sum_{\substack{i+j+k=N+2}} \frac{1}{2} D^{1,0}_{F}  \Bigl({\varphi_i}\lrcorner \bigl({\varphi_j}\lrcorner \Omega_k \bigr)\Bigr) = 0, \\[0.8em]
				\bp_F \bigl(a_{N+2}\w\Omega_0\bigr)=-\sum_{\substack{i+j=N+2}}  \tr R_E\wedge\bigl(\varphi_i\l\Omega_j\bigr)-\sum_{\substack{i+j+k=N+2}} \partial a_i\wedge\bigl(\varphi_j\l\Omega_k\bigr)-\sum_{\substack{i+j=N+2,\,  j\geq 1}}\bp a_i\wedge \Omega_j.
			\end{dcases}
		\end{equation}
		For the first equation  in \eqref{log 0.6}, 
		\allowdisplaybreaks
		\begin{align*}
			-\bar\partial_F \sum_{i+j=N+1} D^{1,0}_F(\varphi_i \lrcorner \Omega_j)
			&= D^{1,0}_F \Bigg(
			\sum_{i=1}^{N+1} \bar\partial \varphi_i \lrcorner \Omega_{N+1-i}
			+ \sum_{i=1}^{N+1} \varphi_i \lrcorner \bar\partial_F \Omega_{N+1-i}
			\Bigg) \\[0.9em]
			&= D^{1,0}_F \Bigg(
			\frac12 \sum_{i=1}^{N+1} \sum_{j=1}^{i-1} [\varphi_j,\varphi_{i-j}] \lrcorner \Omega_{N+1-i} \\
			&\qquad\qquad
			- \sum_{i=1}^{N+1} \varphi_i \lrcorner D^{1,0}_F \bigg(
			\sum_{j=1}^{N+1-i} \varphi_j \lrcorner \Omega_{N+1-i-j}
			\bigg)
			\Bigg) \\[0.9em]
			&= D^{1,0}_F \Bigg(
			\frac12 \sum_{i=1}^{N+1} \sum_{j=1}^{i-1} \bigg(
			- D^{1,0}_F \bigl( \varphi_j \lrcorner (\varphi_{i-j} \lrcorner \Omega_{N+1-i}) \bigr) \\
			&\qquad\qquad\qquad
			+ \varphi_j \lrcorner D^{1,0}_F \bigl( \varphi_{i-j} \lrcorner \Omega_{N+1-i} \bigr)
			+ \varphi_{i-j} \lrcorner D^{1,0}_F \bigl( \varphi_j \lrcorner \Omega_{N+1-i} \bigr)
			\bigg) \\
			&\qquad
			- \sum_{i=1}^{N+1} \varphi_i \lrcorner D^{1,0}_F \bigg(
			\sum_{j=1}^{N+1-i} \varphi_j \lrcorner \Omega_{N+1-i-j}
			\bigg)
			\Bigg) \\[0.9em]
			&= D^{1,0}_F \Bigg(
			\sum_{1 \le j < i \le N+1} \varphi_j \lrcorner D^{1,0}_F \bigl( \varphi_{i-j} \lrcorner \Omega_{N+1-i} \bigr) \\
			&\qquad
			- \sum_{i=1}^{N+1} \varphi_i \lrcorner D^{1,0}_F \bigg(
			\sum_{j=1}^{N+1-i} \varphi_j \lrcorner \Omega_{N+1-i-j}
			\bigg)
			\Bigg) \\[0.6em]
			&= 0,
		\end{align*}
		where the third equality holds in view of   Lemma \ref{Tian-Todorov}.
		We then can   \begin{equation}\label{exact1'}
			\textrm{solve for } \Omega_{N+1}
		\end{equation}
		by Lemma \ref{Fujiki}.

		\medskip\medskip
		Moreover, for the second  equation  in \eqref{log 0.6}, 
		\begin{align*}
			&\quad \bar{\partial}_F \sum_{\substack{i+j+k=N+2}} D^{1,0}_{F}  \Bigl({\varphi_i}\lrcorner \bigl({\varphi_j}\lrcorner \Omega_k \bigr)\Bigr) \\
			&= \sum_{i+j+k=N+2} \bar{\partial}_F \Bigl(
			- [\varphi_j, \varphi_i] \lrcorner \Omega_k
			+ \varphi_i \lrcorner D^{1,0}_F \bigl(\varphi_j \lrcorner \Omega_k\bigr)
			+ \varphi_j \lrcorner D^{1,0}_F \bigl(\varphi_i \lrcorner \Omega_k\bigr)
			\Bigr) \\[0.6em]
			&= - \sum_{i+j+k=N+2} \bar{\partial}_F \bigl([\varphi_j, \varphi_i] \lrcorner \Omega_k\bigr)
			- 2 \sum_{\substack{j+k=N+2\\ k \ge 1}} \bar{\partial}_F \bigl(\varphi_j \lrcorner \bar{\partial}_F \Omega_k\bigr) \\[6pt]
			&= - \sum_{k=0}^{N+2} \Biggl( \bar{\partial} \sum_{i+j=N+2-k} [\varphi_j, \varphi_i] \Biggr) \lrcorner \Omega_k \\[3pt]
			&\quad + 2 \sum_{k=1}^{N+2} \Biggl( \frac{1}{2} \sum_{i+j=N+2-k} [\varphi_j, \varphi_i]
			- \bar{\partial} \varphi_{N+2-k} \Biggr) \lrcorner \bar{\partial}_F \Omega_k \\[6pt]
			&= 0,
		\end{align*}
		where the first equality is due to Lemma \ref{Tian-Todorov}, and the last equality uses \eqref{log 0.4} and \eqref{log 0.5}. Consequently, we can also $$\textrm{solve for $\sum_{\substack{i+j=N+2}} \varphi_i \lrcorner \Omega_j $,  and thus  for $\varphi_{N+2}$}.$$  Moreover,  we may ensure that
		\begin{equation}\label{exact 2'}
			\sum_{\substack{i+j=N+2}} \varphi_i \lrcorner \Omega_j = D_F^{1,0} B_{N+2}
		\end{equation}
		for some $B_{N+2}\in A^{n-2,1}(X,F)$ 	by Lemma \ref{Fujiki}.

		\medskip
		\medskip
		Finally, for the third equation in \eqref{log 0.6}, since we have already solved for $\varphi_k$ for $2\le k\le N+2$ and $a_k$ for $2\le k\le N+1$, its right-hand side is $\bar\partial_F$-closed. Moreover, using the solvability of \eqref{log 0.2} for all $k\le N$, together with \eqref{exact 2'}, \eqref{bianchi2}, \eqref{exact1'}, and the Leibniz rule, we obtain
		\begin{equation}
			\begin{aligned}
				\bar{\partial}_F \bigl(a_{N+2}\wedge \Omega_0\bigr)
				&=-\sum_{\substack{i+j=N+2,\,  j\geq 1}}\bar\partial a_i\wedge \Omega_j
				- \operatorname{tr} R_E \wedge \bigl(D_F^{1,0} B_{N+2}\bigr)
				-\sum_{\substack{i+j+k=N+2}} \partial a_i\wedge\bigl(\varphi_j\lrcorner\Omega_k\bigr) \\[0.5em]
				&=-\sum_{\substack{i+j=N+2,\,  j\geq 1}}\bar\partial a_i\wedge \Omega_j
				- D_F^{1,0}\Biggl( \operatorname{tr} R_E \wedge B_{N+2} -  \sum_{\substack{i+j+k=N+2}}  a_i\wedge\bigl(\varphi_j\lrcorner\Omega_k\bigr)\Biggr)\\[0.2em]
				&\quad+ \sum_{\substack{i+j=N+2,\, j\geq 1}}a_i \wedge \bar\partial_F \Omega_j \\[0.5em]
				&= -\sum_{\substack{i+j=N+2,\, j\geq 1}} \bar\partial_F (a_i \wedge \Omega_j)
				- D_F^{1,0}\Biggl( \operatorname{tr} R_E \wedge B_{N+2} -  \sum_{\substack{i+j+k=N+2}}  a_i\wedge\bigl(\varphi_j\lrcorner\Omega_k\bigr)\Biggr).
			\end{aligned}
		\end{equation}
		Then Lemma \ref{Fujiki} applies to \begin{equation}\label{exact 3'}
			\textrm{solve for} \sum_{\substack{i+j=N+2}} a_i\wedge \Omega_j, \textrm{and hence for }  a_{N+2}.
		\end{equation}
		
		\medskip
		\medskip
		\medskip
		In summary, for given $\Omega_0, \varphi_1$ and $A_1$, we have constructed $\Omega_i$ for $1\le i\le N+1$, together with $\varphi_i$ and $a_i$ for $2\le i\le N+2$, such that \eqref{log 0.2}$_{N+1}$ holds. Therefore,  the system \eqref{log 0.2} admits a formal recursive solution. This completes the proof of Theorem \ref{main thm}.
	\end{proof}

	\medskip
	\medskip
	\begin{proof}[{{Proof of Theorem \ref{main thm2}}}]
		In the $K$-trivial Calabi--Yau case,  $F\simeq\mathcal{O}_X$ in \eqref{bundle F}. Therefore, we only need to verify the relevant weak $\partial\bar\partial$-lemmata of certain types. More precisely, by \eqref{exact1} and \eqref{exact1'} (resp.  \eqref{exact2} and \eqref{exact 2'}; resp.  \eqref{exact3} and \eqref{exact 3'}), this amounts to requiring that $X$ satisfies $\mathbb{S}^{n,1}$ (resp.  $\mathbb{B}^{n-1,2}$; resp. $\mathbb{S}^{n,2}$).
	\end{proof}

	\medskip
	\medskip
	\subsection{\(\mathbb{P}^1\)-bundle  proof of  Theorem \ref{main thm}}
	\label{second}
	
	\begin{proof}	
		Following the  beautiful strategy of Iacono--Manetti \cite{IM21}, we give an
		alternative proof of Theorem \ref{main thm}. The  \textit{only new}
		ingredient is Step \ref{step 1'}, where we  in particular prove the Fujiki case of the unobstructedness of
		locally trivial deformations of log pairs \((Y,\Delta)\) 
		with \(K_Y\otimes\mathcal{O}_Y(\Delta)\). Step \ref{step 2'} then applies the
		\(\mathbb P^1\)-bundle construction of Iacono--Manetti to deduce the
		unobstructedness of \((X,L)\), while Step \ref{step 3'} passes from
		\(L=\det E\) to \(E\). The last two steps are recalled from the cited
		references, and no originality is claimed for them.
		\setcounter{step}{0}
		
		\begin{step}{A $\bar\partial$-equation for twisted logarithmic forms: Fujiki version}\label{step 1'}
		\end{step}
		In this step, we extend the (weak) logarithmic
		$\partial\bar\partial$-lemma of \cite{Zha26}, together with its main
		applications, from compact K\"ahler manifolds to its \textit{bimeromorphic setting}.
		\begin{theorem}[{Fujiki version of \cite[Theorem  B]{Zha26}}]
			\label{dbar-Fujiki}
			Let $X$ be an $n$-dimensional Fujiki manifold, and let
			$D=\sum_{i=1}^sD_i$ 
			be a simple normal crossing divisor on $X$. Let $L\to X$ be a
			holomorphic line bundle such that
			\begin{equation}\label{eq:weighted-line-bundle}
				L
				=
				\sum_{i=1}^sq_iD_i
				\quad
				\text{in }
				\operatorname{Pic}_{\mathbb Q}(X),
				\quad
				q_i\in[0,1]\cap\mathbb Q.
			\end{equation}
			Here the formulation \eqref{eq:weighted-line-bundle} means that, for some positive
			integer $N$,
			\begin{equation*}
				L^{\otimes N}
				\simeq
				\mathcal O_X
				\left(
				\sum_{i=1}^sNq_iD_i
				\right).
			\end{equation*}
			Then for every
			$\alpha
			\in
			A^{0,q}\left(
			X,\Omega_X^{p-1}(\log D)\otimes L^\vee
			\right)$
			satisfying
			\begin{equation*}
				\bar\partial D'_{h_L^*}\alpha=0
				\quad
				\text{on }
				X^\circ:=X\setminus\operatorname{Supp}D,
			\end{equation*}
			there exists
			$	\chi
			\in
			A^{0,q-1}\left(
			X,\Omega_X^p(\log D)\otimes L^\vee
			\right)$
			such that
			\begin{equation}\label{ede}
				\bar\partial\chi
				=
				D'_{h_L^*}\alpha
				\quad
				\text{on }X^\circ.
			\end{equation}
			Here  $D^\prime_{h_L^*}$ represents the $(1,0)$-part of the integrable logarithmic connection  along $D$, induced by the singular metric $h_L^*$ on the dual bundle $L^\vee$.
		\end{theorem}
		
		\begin{proof}
			Choose a proper modification
			$\pi:\widetilde X\rightarrow X$
			such that $\widetilde X$ is compact K\"ahler and
			$	\widetilde D
			:=
			\left(\pi^{-1}D\right)_{\mathrm{red}}
			$
			is a simple normal crossing divisor.  
			Write
			\begin{equation*}
				\pi^*D
				=
				\sum_{i=1}^sD_i'
				+
				\sum_{j=1}^rm_jW_j,
				\quad
				m_j\in\mathbb Z_{\geq1},
			\end{equation*}
			where $D_i'$ is the strict transform of $D_i$, and 
			$W_1,\ldots,W_r$ are the exceptional prime divisors contained in
			$\pi^{-1}(D)$.
			Then
			\begin{equation*}
				\widetilde D
				=
				\sum_{i=1}^sD_i'
				+
				\sum_{j=1}^rW_j.
			\end{equation*}
			Pulling back $L$ then gives
			\begin{equation*}
				\widetilde L:=\pi^*L
				=
				\sum_{i=1}^sq_iD_i'
				+
				\sum_{j=1}^rp_jW_j
				\quad
				\text{in }
				\operatorname{Pic}_{\mathbb Q}(\widetilde X),
			\end{equation*}
			where
			$
			0\leq p_j\leq m_j
			$
			since
			$0\leq q_i\leq1$.
			
			For every $j$, define
			\begin{equation*}
				k_j
				:=
				\begin{cases}
					0, & p_j=0,\\[4pt]
					\lceil p_j\rceil-1, & p_j>0,
				\end{cases}
				\quad
				\widehat q_j:=p_j-k_j,
			\end{equation*}
			and set
			\begin{equation*}
				Q:=\sum_{j=1}^rk_jW_j,
				\quad
				\widehat L
				:=
				\widetilde L\otimes
				\mathcal O_{\widetilde X}(-Q).
			\end{equation*}
			Then  it is easy to check that the following two facts hold:
			\begin{equation}\label{eq:weighted-hat-L}
				\widehat L
				=
				\sum_{i=1}^s q_i D_i'
				+
				\sum_{j=1}^r \widehat q_j W_j
				\quad
				\text{in }
				\operatorname{Pic}_{\mathbb Q}(\widetilde X),
			\end{equation}
			with all $q_i, \widehat q_j \in [0,1] \cap \mathbb Q$, and that
			\begin{equation}\label{eq:divisor-inequality}
				\pi^* D \geq \widetilde D + Q.
			\end{equation}
			
			\medskip
			\medskip
			Consider the following morphism	
			\begin{equation*}\label{eq:Phi}
				\Phi^{p,q}:
				H^q \Big
				(X,\Omega_X^p(\log D)\otimes L^\vee
				\Big)
				\longrightarrow
				H^q\left(
				\widetilde X,
				\Omega_{\widetilde X}^p(\log\widetilde D)
				\otimes{\widehat L}^\vee
				\right),
				\quad
				[\beta]\longmapsto[s_Q\pi^*\beta],
			\end{equation*}
			where $s_Q
			\in
			H^0\left(
			\widetilde X,\mathcal O_{\widetilde X}(Q)
			\right)$ is the canonical section.
			\textit{We claim that $\Phi^{p,q}$ is injective.}

			\medskip
			Indeed, logarithmic Serre duality gives the pairing
			\begin{align*}
				&
				H^q\Big(
				X,\Omega_X^p(\log D)\otimes L^\vee
				\Big)
				\times
				H^{n-q}\Big(
				X,\Omega_X^{n-p}(\log D)(-D)\otimes L
				\Big)
				\longrightarrow
				\mathbb C.
			\end{align*}
			Suppose that $\Phi^{p,q}([\beta])=0$. Then
			$
			s_Q\pi^*\beta=\bar\partial u
			$
			for some suitable form $u$. Let
			$				\gamma
			\in
			A^{0,n-q}\left(
			X,
			\Omega_X^{n-p}(\log D)(-D)\otimes L
			\right)
			$
			be $\bar\partial$-closed. By
			\eqref{eq:divisor-inequality}, $\pi^*\gamma$ vanishes along
			$Q$, and hence
			\begin{equation*}
				\widehat\gamma
				:=
				\frac{\pi^*\gamma}{s_Q}
			\end{equation*}
			is a well-defined $\bar\partial$-closed form with values in the bundle 
			$		\Omega_{\widetilde X}^{n-p}(\log\widetilde D)
			(-\widetilde D-Q)\otimes\widetilde L,
			$
			which  is the Serre dual of
			$
			\Omega_{\widetilde X}^p(\log\widetilde D)
			\otimes{\widehat L}^\vee.
			$
			Since $\pi$ is a biholomorphism outside an analytic subset (with zero Lebesgue measure), Stokes' theorem gives
			\begin{align*}
				\int_X\beta\wedge\gamma
				=
				\int_{\widetilde X}
				(s_Q\pi^*\beta)\wedge\widehat\gamma
				=
				\int_{\widetilde X}
				\bar\partial (u\wedge\widehat\gamma)
				=
				0.
			\end{align*}
			Thus $[\beta]$ pairs trivially with every dual class, thus $[\beta]=0$. As a result,
			$\Phi^{p,q}$ is injective.

			We now set
			$\beta:=D^\prime_{h_L^*}\alpha$, which is  $\bar\partial$-closed by assumption. 
			By direct verifications, 
			\begin{equation*}
				D^\prime_{h_{\widehat L}^*}
				(s_Q\pi^*\alpha)
				=
				s_Q\pi^*
				\left(
				D^\prime_{h_L^*}\alpha
				\right)
				=
				s_Q\pi^*\beta.
			\end{equation*}
			By \eqref{eq:weighted-hat-L}, we may apply
			\cite[Theorem  B]{Zha26} on 
			$\widetilde X$ to find some
			$	\widehat\chi
			\in
			A^{0,q-1}\left(
			\widetilde X,
			\Omega_{\widetilde X}^p(\log\widetilde D)
			\otimes{\widehat L}^ \vee
			\right)$
			such that
			\begin{equation*}
				\bar\partial\widehat\chi
				=
				D^\prime_{h_{\widehat L}^*}
				(s_Q\pi^*\alpha)
				=
				s_Q\pi^*\beta.
			\end{equation*}
			Therefore
			$
			\Phi^{p,q}([\beta])=0,$
			and thus $[\beta]=0$, so
			\eqref{ede} is solvable. This completes the proof.
		\end{proof}

		With Theorem  \ref{dbar-Fujiki} in hand, the proof of \cite[Theorem A]{Zha26}, which follows the spirit of \cite{KKP08} (see \cite[\S 5.4.3]{Zha26}) applies verbatim to yield:
		
		\begin{theorem}[{Fujiki version of \cite[Theorem A]{Zha26}}]\label{thm-fujiki-inobstr}
			Let $(X,D)$ be as in Theorem  \ref{dbar-Fujiki}, and denote by $K_X$ the canonical bundle of $X$. Assume that 
			\begin{equation*}\label{qwe}
				-K_X = \sum_{i=1}^s q_i D_i \quad	\text{in }
				\operatorname{Pic}_{\mathbb Q}(X),
				\quad
				q_i\in[0,1]\cap\mathbb Q.
			\end{equation*}
			Then the DGLA
			\[
			(A^{0,\bullet}(X, T_X(-\log D)), \bar\partial),
			\]
			which controls the locally trivial deformations of the pair $(X,D)$, is homotopy abelian.
			In particular, the locally trivial deformations of the pair $(X,D)$ are unobstructed.
		\end{theorem}
		Here recall that two DGLAs $\mathbf L$ and $\mathbf M$ are said to be \textit{homotopy equivalent} if there exists a quasi-isomorphism between them. In particular, a DGLA is said to be \textit{homotopy abelian} if it is homotopy equivalent to an abelian DGLA (i.e., one with trivial bracket).

		\medskip
		\medskip
		As another application of Theorem  \ref{dbar-Fujiki}, which may be of
		independent interest, the same argument of \cite[\S 5.3]{Zha26} extends our K\"ahler version of Ambro's injectivity
		theorem \cite{Amb14,Zha26} to the Fujiki setting.
		This provides a new analytic proof of Fujino's result
		\cite[Theorem 1.2]{Fuj17}, without using mixed Hodge theory for
		cohomology with compact support.
		\begin{theorem}[{Fujino}]\label{Fujino}
			Let $(X,D)$ be as in Theorem \ref{dbar-Fujiki},   and let $L$ be a holomorphic line bundle over $X$ such that  \begin{equation*}\label{satis}
				L-K_X=\sum_{i=1}^{s}q_iD_i\quad
				\text{in }
				\operatorname{Pic}_{\mathbb Q}(X),
				\quad
				q_i\in(0,1]\cap\mathbb Q.
			\end{equation*}
			Then the restriction map
			\begin{equation*}\label{VVV}
				H^q(X,L)\stackrel{{j}}{\longrightarrow} H^q(X^\circ,L|_{X^\circ})
			\end{equation*}
			is  injective, for all $q$. 
		\end{theorem}
		\begin{step}{Rank one case}\label{step 2'}
		\end{step}
		Let $L\rightarrow X$ be a holomorphic line bundle over a  Calabi--Yau Fujiki manifold. Following the approach of \cite[$\S$2]{IM21}, consider the $\mathbb P^1$-bundle:
		$$p: Y=\mathbb P(\mathcal{O}_X\oplus L)\longrightarrow X.$$
		Then $Y$ is also a Fujiki manifold. Indeed, let $\mu\colon \widetilde X \to X$ be a Kähler modification. We have the following commutative diagram:
		\[
		\begin{tikzcd}
			\mathbb P(\mu^*(\mathcal O_X \oplus L)) \ar[d] \ar[r, "g"] & \mathbb P(\mathcal O_X \oplus L)=Y \ar[d] \\
			\widetilde X \ar[r, "\mu"'] & X.
		\end{tikzcd}
		\]
		Since
		\(
		\mathbb P(\mu^*(\mathcal O_X \oplus L)) = \widetilde X \times_X Y,
		\)
		the morphism $g$ is also bimeromorphic. Thus $Y$ is bimeromorphic to a compact K\"ahler manifold.
		Denote by $$\mathbb{V}(L^*)=\operatorname{Spec}_{\mathcal{O}_X}\big(\textrm{Sym}(L^*)\big)\longrightarrow X$$ the total space of the line bundle $L$. Take $s_0$ (resp. $s_\infty$) as the zero section of $\mathbb{V}(L^*)$ (resp. the section at infinity of $Y$), which corresponds to the divisor $\Delta_0$ (resp. $\Delta_{\infty}$). Then 
		$$
		Y=\mathbb{V}(L^*)\cup \Delta_\infty.
		$$
		Set  $$\Delta=\Delta_0+\Delta_{\infty},$$
		which is then a smooth divisor on $Y$. 
		
		By using the relative Euler exact sequence \cite[$\S$III, Exercise 8.4]{Har77} (which still applies in the analytic category), we obtain
		$$K_Y \otimes \mathcal O_Y(\Delta) \simeq p^*K_X.$$
		In the case where $K_X$ is torsion, we may apply Theorem \ref{thm-fujiki-inobstr} with all $q_i=1$ to the pair $(Y,D)$ to see that the DGLA,
		$(A^{0,\bullet}(Y, T_Y(-\log D)), \bar\partial),$
		which controls the locally trivial deformations of the  log Calabi--Yau pair $(Y,\Delta)$, is homotopy abelian. By \cite[Lemma 2.2]{IM21}, there exists a morphism
		$$\Phi \colon \mathcal D(X,L) \longrightarrow p_* (T_Y(-\log D))$$
		inducing a quasi-isomorphism of DGLAs
		$$\Phi \colon R\Gamma(X,\mathcal D(X,L)) \longrightarrow (A^{0,\bullet}(Y, T_Y(-\log D)), \bar\partial),$$
		Here,   $\mathcal D(X,L)$ denotes the sheaf of derivations of pairs as in  \cite[$\S$5]{IM19}, which  is the subsheaf of
		\(
		\operatorname{Der}_{\mathbb C}(\mathcal O_X,\mathcal O_X) \times \operatorname{Hom}_{\mathbb C}(L,L)
		\)
		consisting of pairs $(h,u)$ satisfying
		\[
		u(ax)=h(a)x + a u(x)
		\]
	for all local sections $a$ of $\mathcal O_X$ and $x$ of $L$.   Moreover, $\mathcal D(X,L) $ is isomorphic to the sheaf $\mathcal D^1(L)$ of scalar differential operators of order at most one on $L$ (and hence to the Atiyah extension $A(L)$; see Remark \ref{rmk-iso}). Also, the DGLA $R\Gamma(X,\mathcal D(X,L))$ controls the deformations of the pair $(X,L)$. We refer the reader to \cite[$\S$2]{IM19} for further details. Consequently,  $R\Gamma(X,\mathcal D(X,L))$ is also homotopy abelian. In particular, $(X,L)$ has unobstructed joint deformations.
		\begin{step}{Higher rank case}\label{step 3'}
		\end{step}
		Theorem \ref{main thm} follows from the following result by taking $L=\det E$ in Step \ref{step 2'}. Note that the statement in the cited references is for projective smooth varieties, but its proof goes through for any compact complex manifold.
		
		\begin{proposition}[{\cite[Theorem 1.2 (ii)]{IM19}; \cite[Proposition 11.16]{Fel25}}]\label{reduction}
			Let $X$ be a compact complex manifold and $E$ a holomorphic vector bundle over $X$. Assume $H^2(X,\End^0 E)=0$, then the natural map
			\[
			\operatorname{Def}(X,E) \longrightarrow \operatorname{Def}(X,\det E)
			\]
			is smooth as a morphism of deformation functors.\footnote{
				Recall that a morphism $\mathcal F\to\mathcal G$
				of deformation functors is \textit{smooth} if, for every
				surjection $A'\twoheadrightarrow A$ of local Artin
				$\mathbb C$-algebras with residue field $\mathbb C$,
				the natural map
				$\mathcal F(A')\to
				\mathcal F(A)\times_{\mathcal G(A)}\mathcal G(A')$
				is surjective; see e.g.
				\cite[Definition  3.2.7]{Man22}.} In particular, if $(X,\det E)$ has unobstructed joint deformations, then so does $(X,E)$.
		\end{proposition}
	\end{proof}

	\subsection{\texorpdfstring{${T^1}$}{T1}-lifting proofs of
		Theorem \ref{main thm} and  the
		\texorpdfstring{${\partial\bar\partial}$}{ddbar}-manifold case of
		Theorem \ref{main thm2}}
	
	\medskip
	
	\begin{proof}[{Proof of Theorem \ref{main thm2} when $X$ is a $K$-trivial  Calabi--Yau $\p\bp$-manifold}]
		
		Building on the sharp observation of Iacono--Manetti
		\cite[Remark 2.6]{IM21}, we show that their \(T^1\)-lifting argument
		also applies in our setting and provide further details for the
		reader's convenience.

		By Proposition \ref{reduction}, it suffices to consider the line bundle case. Set $L=\det E$.	Put the Artinian base \(A_m=\mathbb C[\epsilon]/(\epsilon^{m+1})\), and let
		\((X_m,L_m)\) be an arbitrary deformation of \((X,L)\) over \(A_m\).  Write
		\(\ell_m(\cdot)\) for the length of an \(A_m\)-module.
		
		By Deligne's Artinian length criterion \cite[(3.5.1)]{Del68},
		if \(C^\bullet\) is a bounded complex of finite free \(A_m\)-modules, then
		\begin{equation}\label{deligne}
			\ell_m H^i(C^\bullet)
			\leq
			(m+1)\dim_{\mathbb C}
			H^i(C^\bullet\otimes_{A_m}\mathbb C).
		\end{equation}
		Equality implies that \(H^i(C^\bullet)\) is free over \(A_m\) and compatible with reduction to $A_{m-1}$; namely, $$H^i(C^\bullet)\otimes_{A_m}A_{m-1}
		\xrightarrow{\simeq}
		H^i(C^\bullet\otimes_{A_m}A_{m-1}).$$
		
		Set
		\[
		H_m^{p,q}:=
		H^q(X_m,\Omega^p_{X_m/A_m}).
		\]
		The relative de Rham cohomology is topologically constant:
		\[
		H^k_{\mathrm{dR}}(X_m/A_m)
		\simeq H^k(X,\mathbb C)\otimes_{\mathbb C}A_m.
		\]
		Consequently, applying \eqref{deligne} to the Hodge--de Rham spectral sequence gives
		\begin{equation}\label{Degeneration}
			(m+1)b_k
			\leq
			\sum_{p+q=k}\ell_m H_m^{p,q}
			\leq
			(m+1)\sum_{p+q=k}h^{p,q}_{\bp}(X).
		\end{equation}
		Since the \(\partial\bar\partial\)-lemma on $X_0$ implies degeneration at \(E_1\),
		\[
		b_k=\sum_{p+q=k}h^{p,q}_{\bp}(X).
		\]
		Thus equality holds throughout \eqref{Degeneration}.  Hence every \(H_m^{p,q}\) is free,
		satisfies base change and the relative Hodge--de Rham spectral
		sequence over  \(A_m\) degenerates at \(E_1\).
		
		In particular, a holomorphic volume form on \(X\) lifts to \(X_m\) and
		trivializes \(K_{X_m/A_m}\).  Contraction with this form gives
		\begin{equation}\label{iso}
			T^{1,0}_{X_m/A_m}\simeq\Omega^{n-1}_{X_m/A_m},
			\quad
			\mathcal O_{X_m}\simeq\Omega^n_{X_m/A_m}.
		\end{equation}

		Consider the relative Atiyah sequence
		\[
		0\longrightarrow\mathcal O_{X_m}
		\longrightarrow A(L_m)
		\longrightarrow T^{1,0}_{X_m/A_m}
		\longrightarrow0.
		\]
		Its cohomology sequence contains
		\[
		H^0(T^{1,0}_{X_m/A_m})
		\xrightarrow{\alpha_{0,m}}H^1(\mathcal O_{X_m})
		\longrightarrow H^1(A(L_m))
		\longrightarrow H^1(T^{1,0}_{X_m/A_m})
		\xrightarrow{\alpha_{1,m}}H^2(\mathcal O_{X_m}),
		\]
		and therefore
		\begin{equation} \label{sequence}
			0\longrightarrow\operatorname{coker}\alpha_{0,m}
			\longrightarrow H^1(A(L_m))
			\longrightarrow\ker\alpha_{1,m}
			\longrightarrow0.
		\end{equation}

		Under \eqref{iso}, the two boundary maps are the Hodge components of cup product
		with \(c_1(L_m)\):
		\begin{equation}\label{chern}
			\alpha_{i,m}\colon
			H_m^{n-1,i}
			\xrightarrow{\;\cup c_1(L_m)\;}
			H_m^{n,i+1},
			\quad i=0,1.
		\end{equation}
		
		\medskip
		\textit{We now use the full strength of the \(\partial\bar\partial\)-lemma, rather than merely \(E_1\)-degeneration}. It identifies \(H_{\bar\partial}^{p,q}(X)\) with the subspace of \(H_{\mathrm{dR}}^{p+q}(X)\) represented by \(d\)-closed forms of pure type \((p,q)\), and gives the decomposition
		\[
		H_{\mathrm{dR}}^k(X,\mathbb C)
		=
		\bigoplus_{p+q=k}H^{p,q}_{\bp}(X).
		\]
		
		\textit{In what follows,  we fix $k=n-1,n$.}     Consider the total cup-product maps
		\begin{equation*}
			T_0^k:
			H_{\mathrm{dR}}^k(X,\mathbb C)
			\longrightarrow
			H_{\mathrm{dR}}^{k+2}(X,\mathbb C)
		\end{equation*}
		and
		\begin{equation*}
			T_m^k:
			H_{\mathrm{dR}}^k(X_m/A_m)
			\longrightarrow
			H_{\mathrm{dR}}^{k+2}(X_m/A_m).
		\end{equation*}
		Since $k+2>n$, every nonzero Hodge component of
		$H_{\mathrm{dR}}^{k+2}(X,\mathbb C)$ has both bidegrees at least
		one. Consequently,
		\begin{equation*}
			H_{\mathrm{dR}}^{k+2}(X,\mathbb C)
			=
			\bigoplus_{p+q=k}
			H_{\bar\partial}^{p+1,q+1}(X),
		\end{equation*}
		and $T_0^k$ decomposes as the direct sum of the maps
		\begin{equation*}
			T_0^{p,q}:
			H_{\bar\partial}^{p,q}(X)
			\longrightarrow
			H_{\bar\partial}^{p+1,q+1}(X).
		\end{equation*}
		It follows that
		\begin{equation}\label{8}
			\dim_{\mathbb C}\ker T_0^k
			+
			\dim_{\mathbb C}\operatorname{coker}T_0^k
			=
			\sum_{p+q=k}
			\left(
			\dim_{\mathbb C}\ker T_0^{p,q}
			+
			\dim_{\mathbb C}\operatorname{coker}T_0^{p,q}
			\right).
		\end{equation}
		
		Under the topological identification of relative de Rham cohomology, \(c_1(L_m)\) corresponds to \(c_1(L)\otimes1\), hence
		\begin{equation}\label{9}
			T_m^k=T_0^k\otimes_{\mathbb C}A_m. 
		\end{equation}
		Moreover, $T_m^k$ raises the relative Hodge filtration by one, and its associated-graded maps  are
		\[
		T_m^{p,q}\colon
		H_m^{p,q}\longrightarrow H_m^{p+1,q+1},
		\quad p+q=k.
		\]
		These include the boundary maps in \eqref{chern}:
		\[
		\alpha_{i,m}=T_m^{n-1,i},
		\quad i=0,1.
		\]
		
		Applying the spectral-sequence length inequality to the filtered two-term complex defined by \(T_m^k\), and Deligne's inequality \eqref{deligne} to the two-term complexes defined by \(T_m^{p,q}\), gives
		\begin{equation*}\label{length-chain}
			\begin{aligned}
				(m+1)\Bigl(
				\dim_{\mathbb C}\ker T_0^k
				+\dim_{\mathbb C}\operatorname{coker}T_0^k
				\Bigr)
				&=
				\ell_m\bigl(\ker T_m^k\bigr)
				+\ell_m\bigl(\operatorname{coker}T_m^k\bigr)\\[0.5em]
				&\le
				\sum_{p+q=k}\Bigl(
				\ell_m\bigl(\ker T_m^{p,q}\bigr)
				+\ell_m\bigl(\operatorname{coker}T_m^{p,q}\bigr)
				\Bigr)\\[0.5em]
				&\le
				(m+1)\sum_{p+q=k}\Bigl(
				\dim_{\mathbb C}\ker T_0^{p,q}
				+\dim_{\mathbb C}\operatorname{coker}T_0^{p,q}
				\Bigr)\\[0.5em]
				&=
				(m+1)\Bigl(
				\dim_{\mathbb C}\ker T_0^k
				+\dim_{\mathbb C}\operatorname{coker}T_0^k
				\Bigr).
			\end{aligned}
		\end{equation*}
		where the first equality follows from \eqref{9}, and the last from \eqref{8}. Hence equality holds throughout.
		
		As a result, 
		\[
		\operatorname{coker}\alpha_{0,m}
		=
		\operatorname{coker}T_m^{n-1,0},
		\quad
		\ker\alpha_{1,m}
		=
		\ker T_m^{n-1,1}
		\]
		are free and compatible with  base change. 
		It now follows from \eqref{sequence} that \(H^1(A(L_m))\) is free over \(A_m\) and that
		\[
		H^1(A(L_m))\otimes_{A_m}A_{m-1}
		\simeq
		H^1(A(L_{m-1})).
		\]
		In particular,
		\[
		H^1(A(L_m))
		\longrightarrow
		H^1(A(L_{m-1}))
		\]
		is surjective. This is precisely the  \textit{\(T^1\)-lifting property} for the pair \((X,L)\)  and hence it  has  unobstructed deformations.

	\end{proof}
	
	\begin{proof}[{Proof of  Theorem \ref{main thm}}]
		We first   record a useful consequence due to Wells. 
		\begin{lemma}[{\cite[Theorem 3.1 (c)]{Wel74}}]\label{wells}
			Let $f: Y \to X$ be a proper surjective morphism between two equidimensional complex manifolds, and let $E$ be a holomorphic vector bundle over $X$. Then $f$ induces injections for all $p$ and $q$:
			\[
			f^*: H^{p,q}_{\bp}(X, E) \hooklongrightarrow H^{p,q}_{\bp}(Y, f^*E).
			\]
		\end{lemma}
		If $X$ is a Fujiki manifold then it is a
		$\p\bp$-manifold. By the above $T^1$-lifting proof, it suffices to
		consider the case in which $K_X$ is torsion. Then there exists a finite
		\'etale cover
		\(\pi:\widetilde{X}\to X\) such that
		\(K_{\widetilde{X}} \simeq \mathcal{O}_{\widetilde{X}}\)
		(see, e.g., \cite{Par91}). 
		Let \(f:Y\to X\) be a proper modification from a compact K\"ahler
		manifold \(Y\), and set
		\(\widetilde{Y}:=Y\times_X\widetilde{X}\).
		We then obtain the commutative diagram
		\begin{equation}\label{diagram}
			\begin{tikzcd}
				\widetilde{Y}
				\arrow[d, "\widetilde{\pi}"']
				\arrow[r, "\widetilde{f}"]
				& \widetilde{X} \arrow[d, "\pi"] \\
				Y \arrow[r, "f"] & X.
			\end{tikzcd}
		\end{equation}
		By stability under base change, \(\widetilde{\pi}\) remains finite
		\'etale, while \(\widetilde{f}\) is bimeromorphic.
		In particular, \(\widetilde{Y}\) is again a compact K\"ahler manifold, therefore \(\widetilde{X}\) is also a Fujiki manifold.

		Now for any \(N\geq 1\), let
		\[
		v_N := \frac{1}{2}\left\{\left[\sum_{i=1}^{N}(A_i,\varphi_i),
		(A_{N+1-i},\varphi_{N+1-i})\right]\right\}
		\;\in\; H^2\bigl(X,A(L)\bigr), \,\,\text{with }L=\det E,
		\]
		be an \(N\)-th order obstruction class.
		Since \(\pi\) is \'etale, it is proper and surjective, Lemma \ref{wells} implies that
		the pull-back map
		\[
		\pi^* \colon H^2\bigl(X,A(L)\bigr)
		\longrightarrow
		H^2\bigl(\widetilde{X},A(\pi^*L)\bigr)
		\]
		is injective, which completes the proof.
	\end{proof}

	\subsection{Proof of Corollary \ref{simple}}
	\begin{proof}
		Since $K_X\simeq\mathcal O_X$ and
		$(\operatorname{End}^0E)^\vee\simeq\operatorname{End}^0E$,
		Serre duality gives
		$$
		H^2(X,\operatorname{End}^0E)
		\simeq
		H^0(X,\operatorname{End}^0E)^\vee
		=0,
		$$
		where the last equality follows from the simplicity of $E$.
		To apply Theorem \ref{main thm2}, it remains to verify the three weak $\p\bp$-conditions. For this, we use the following observation. 
		\begin{proposition}\label{prop:weak-surface}
			Let \(X\) be a compact complex surface. Then \(X\) always
			satisfies \(\mathbb S^{2,2}\) and \(\mathbb B^{1,2}\).
			Moreover, \(X\) satisfies \(\mathbb S^{2,1}\)
			if and only  \(X\) is K\"ahler.
		\end{proposition}
		
		\begin{proof}
			We first recall the cohomological characterization of the weak
			conditions. The Bott--Chern and Aeppli cohomology groups are
			defined by
			\begin{equation*}
				H_{\mathrm{BC}}^{p,q}(X)
				:=
				\frac{
					\ker\partial\cap\ker\bar\partial
				}{
					\operatorname{im}\partial\bar\partial
				},
				\qquad
				H_{\mathrm A}^{p,q}(X)
				:=
				\frac{
					\ker\partial\bar\partial
				}{
					\operatorname{im}\partial+\operatorname{im}\bar\partial
				}.
			\end{equation*}
			By \cite[p.  3009]{RZ18}, the condition
			\(\mathbb S^{p,q}\) is equivalent to the injectivity of
			\begin{equation}\label{eq:characterization-S}
				\iota_{\bar\partial,\mathrm A}^{p,q}:
				H_{\bar\partial}^{p,q}(X)
				\longrightarrow
				H_{\mathrm A}^{p,q}(X),
			\end{equation}
			whereas \(\mathbb B^{p,q}\) is equivalent to the injectivity of
			\begin{equation}\label{eq:characterization-B}
				\iota_{\mathrm{BC},\partial}^{p,q}:
				H_{\mathrm{BC}}^{p,q}(X)
				\longrightarrow
				H_{\partial}^{p,q}(X).
			\end{equation}

			We first verify \(\mathbb S^{2,2}\). Let
			\(\partial g\in A^{2,2}(X)\). Since
			\(\bar\partial g=0\) for dimensional reasons, Stokes' theorem
			gives
			\begin{equation*}
				\int_X\partial g=\int_Xdg=0.
			\end{equation*}
			By Serre duality, integration identifies
			\(H_{\bar\partial}^{2,2}(X)\) with \(\mathbb C\). Hence
			\([\partial g]_{\bar\partial}=0\), proving
			\(\mathbb S^{2,2}\).
			
			Second, $\mathbb B^{1,2}$ always holds. Indeed, by \cite[Theorem  1.1 (i)]{ADT16}, the natural map\begin{equation*}H_{\mathrm{BC}}^{2,1}(X)\longrightarrow H_{\bar\partial}^{2,1}(X)\end{equation*}is injective for every compact complex surface. Taking complex conjugates then gives the injectivity of $H_{\mathrm{BC}}^{1,2}(X)\longrightarrow H_{\partial}^{1,2}(X)$ and hence $X$ satisfies $\mathbb B^{1,2}$.
			
			It remains to characterize the condition $\mathbb S^{2,1}$.
			Suppose that $X$ satisfies $\mathbb S^{2,1}$. By
			\eqref{eq:characterization-S} and duality, the natural map
			\begin{equation*}
				H_{\mathrm{BC}}^{0,1}(X)
				\longrightarrow
				H_{\bar\partial}^{0,1}(X)
			\end{equation*}
			is surjective. This map is always injective. Indeed, if
			$\alpha=\bar\partial f$ represents a Bott--Chern class, then
			$\partial\alpha=0$ gives $\partial\bar\partial f=0$. Since every
			global pluriharmonic function on a compact complex manifold is
			constant, we have $\alpha=0$. Consequently,
			\begin{equation}\label{eq:BC01-bar01}
				h_{\mathrm{BC}}^{0,1}(X)
				=
				h_{\bar\partial}^{0,1}(X).
			\end{equation}
			
			On the other hand, the Fr\"olicher spectral sequence of every compact
			complex surface degenerates at $E_1$
			\cite[Theorem  IV.2.8]{BHPV04}. In particular, every holomorphic
			one-form is $d$-closed, and hence
			\begin{equation}\label{eq:BC10-bar10}
				h_{\mathrm{BC}}^{1,0}(X)
				=
				h_{\bar\partial}^{1,0}(X).
			\end{equation}
			Complex conjugation for Bott--Chern cohomology gives
			$
			h_{\mathrm{BC}}^{1,0}(X)
			=
			h_{\mathrm{BC}}^{0,1}(X).
			$
			Combining this equality with
			\eqref{eq:BC01-bar01} and \eqref{eq:BC10-bar10}, we obtain
			\begin{equation*}
				h_{\bar\partial}^{1,0}(X)
				=
				h_{\bar\partial}^{0,1}(X).
			\end{equation*}
			
			The $E_1$-degeneration also gives
			\begin{equation*}
				b_1(X)
				=
				h_{\bar\partial}^{1,0}(X)
				+
				h_{\bar\partial}^{0,1}(X)
				=
				2h_{\bar\partial}^{0,1}(X).
			\end{equation*}
			Thus $b_1(X)$ is even, and therefore $X$ is K\"ahler by
			\cite[Theorem  11]{Buchdahl} or 	\cite[Theorem  IV.3.1]{BHPV04}.
			
			Conversely, if $X$ is K\"ahler, then it satisfies the
			$\partial\bar\partial$-lemma and hence $\mathbb S^{2,1}$.
		\end{proof}
		The proof of Corollary \ref{simple} is therefore completed.
	\end{proof}
	
	\subsection{Examples satisfying the three weak $\p\bp$-conditions  without $E_1$-degeneration}
	
	We now give two examples illustrating different features of the three
	weak conditions in Theorem  \ref{main thm2}. Both have trivial canonical
	bundle and satisfy these three conditions, but their Fr\"olicher
	spectral sequences do not degenerate at $E_1$; in particular, neither
	satisfies the full $\p\bp$-lemma.
	\begin{itemize}
		\item[$\centerdot$]
		The first is a compact complex threefold $X$ for which
		$H^2(X,A(E))=0$ for every holomorphic vector bundle $E\rightarrow X$
		satisfying $H^2(X,\operatorname{End}^{0}E)=0$.
		Thus unobstructedness already follows from the vanishing
		of the obstruction space.

		\medskip
		\item[$\centerdot$]
		The second is a compact complex sixfold $X$ admitting a
		simple holomorphic vector bundle $E$ of rank $2$ such that
		$H^2(X,\operatorname{End}^{0}E)=0$, whereas
		\begin{equation*}
			H^2(X,\operatorname{End}E)\simeq\mathbb C,
			\quad
			H^2(X,A(E))\simeq\mathbb C^2.
		\end{equation*}
		Moreover, the forgetful morphism
		$\operatorname{Def}(X,E)\to\operatorname{Def}(X)$
		is not smooth.

		Thus our Tian--Todorov argument proves unobstructedness under
		the three weak $\p\bp$-conditions, even when the  arguments
		based on obstruction-space vanishing or smoothness of the
		forgetful morphism are unavailable.
	\end{itemize}

	\medskip
	
	We verify the weak conditions in these examples using \eqref{eq:characterization-S} and \eqref{eq:characterization-B}.
	\begin{proposition}\label{prop:weak-ddbar-threefold}
		There exists a compact complex threefold $X$ such that:
		\begin{enumerate}[{(1)}]
			\item $K_X\simeq\mathcal O_X$;
			\item \label{2222} $X$ satisfies
			$\mathbb S^{3,1}$, $\mathbb S^{3,2}$ and
			$\mathbb B^{2,2}$;
			\item \(
			E_1^{\bullet,\bullet}(X)
			\not\simeq
			E_\infty^{\bullet,\bullet}(X),
			\)
			then in particular, \(X\) does not satisfy the full
			\(\partial\bar\partial\)-lemma;
			\item  \label{4444} for any holomorphic vector bundle $E\rightarrow  X$  with $H^2(X,\operatorname{End}^{0}E)=0$, one has  $H^2(X, A(E))=0$.
		\end{enumerate}
	\end{proposition}
	
	\begin{proof}
		Consider the connected complex Lie group
		\[
		G=\mathrm{SL}_2(\mathbb C).
		\]
		By examples due to Ghys \cite{Ghy95} and Winkelmann \cite{Win95},
		recalled in \cite[Introduction]{KS23}, one may choose a cocompact
		lattice $\Gamma\subset G$ such that
		\[
		X:=\Gamma\backslash G
		\]
		is a page-$1$-$\partial\bar\partial$-manifold; see
		\cite[Definition  2.10]{PSU24} for the terminology.
		
		Choose a basis of left-invariant holomorphic one-forms
		$\omega_1,\omega_2,\omega_3$ on $X$ satisfying
		\begin{equation}\label{eq:sl2-structure}
			\partial\omega_1
			=
			\omega_2\wedge\omega_3,
			\qquad
			\partial\omega_2
			=
			\omega_3\wedge\omega_1,
			\qquad
			\partial\omega_3
			=
			\omega_1\wedge\omega_2.
		\end{equation}
		
		\begin{enumerate}[{(1)}]
			\item
			\textit{First, $X$ is $K$-trivial Calabi--Yau.}
			Indeed, the form
			\begin{equation*}
				\Omega_X
				:=
				\omega_1\wedge\omega_2\wedge\omega_3
			\end{equation*}
			is nowhere vanishing. Hence $K_X\simeq\mathcal O_X$.
			
			\medskip
			\item
			\textit{We next verify the three weak
				$\partial\bar\partial$-conditions.}
			Put
			\begin{equation*}
				\mathrm P_X^\bullet
				:=
				\Lambda^\bullet
				\langle\omega_1,\omega_2,\omega_3\rangle
			\end{equation*}
			and consider the double complex of left-invariant forms
			\begin{equation*}
				\mathcal I_X^{p,q}
				:=
				\mathrm P_X^p
				\otimes
				\overline{\mathrm P_X^q}.
			\end{equation*}
			By Akhiezer's theorem \cite{Akh97} and the proof of
			\cite[Corollary  13]{KS23}, the inclusion
			$	\mathcal I_X^{\bullet,\bullet}
			\hookrightarrow
			A^{\bullet,\bullet}(X)
			$
			induces isomorphisms in Dolbeault cohomology. (We remark that the choice of $\Gamma$ is essential here: for an arbitrary lattice,
			left-invariant forms need not compute the Dolbeault cohomology.)  By complex
			conjugation, it also induces isomorphisms in
			$\partial$-cohomology.   \cite[Corollary  5]{KS23} then
			implies that the same inclusion induces isomorphisms in
			Bott--Chern and Aeppli cohomology. Therefore the natural maps
			used to verify $\mathbb S^{p,q}$ and $\mathbb B^{p,q}$ may
			be computed on $\mathcal I_X^{\bullet,\bullet}$.
			
			The structure equations show that
			\begin{equation*}
				\partial:
				\mathrm P_X^1
				\longrightarrow
				\mathrm P_X^2
			\end{equation*}
			is an isomorphism and that $\partial\mathrm P_X^2=0$. Indeed,
			\begin{equation*}
				\partial\mathrm P_X^1
				=
				\left\langle
				\omega_2\wedge\omega_3,\,
				\omega_3\wedge\omega_1,\,
				\omega_1\wedge\omega_2
				\right\rangle
				=
				\mathrm P_X^2,
			\end{equation*}
			while \eqref{eq:sl2-structure} directly gives
			$\partial\mathrm P_X^2=0$. Consequently,
			\begin{equation}\label{eq:sl2-small-cohomology}
				H^1(\mathrm P_X^\bullet,\partial)=0,
				\quad
				H^2(\mathrm P_X^\bullet,\partial)=0,
				\quad
				H^3(\mathrm P_X^\bullet,\partial)
				=
				\mathbb C[\Omega_X].
			\end{equation}
			
			It follows from \eqref{eq:sl2-small-cohomology} and complex
			conjugation that
			\begin{equation*}
				H_{\bar\partial}^{3,1}(X)=0,
				\quad
				H_{\bar\partial}^{3,2}(X)=0.
			\end{equation*}
			Thus
			$\iota_{\bar\partial,\mathrm A}^{3,1}$ and
			$\iota_{\bar\partial,\mathrm A}^{3,2}$ are injective, and
			hence $Y$ satisfies $\mathbb S^{3,1}$ and
			$\mathbb S^{3,2}$.
			
			It remains to check $\mathbb B^{2,2}$. Since
			$\partial\mathrm P_X^2=0$ and
			$\bar\partial\overline{\mathrm P_X^2}=0$, every element of
			$\mathrm P_X^2\otimes\overline{\mathrm P_X^2}$ is both
			$\partial$-closed and $\bar\partial$-closed. Moreover,
			\begin{equation*}
				\partial\mathrm P_X^1=\mathrm P_X^2,
				\quad
				\bar\partial\overline{\mathrm P_X^1}
				=
				\overline{\mathrm P_X^2}.
			\end{equation*}
			We therefore obtain
			\begin{align*}
				H_{\mathrm{BC}}^{2,2}(X)
				\simeq
				H_{\mathrm{BC}}^{2,2}(\mathcal I_X)=
				\frac{
					\mathrm P_X^2
					\otimes
					\overline{\mathrm P_X^2}
				}{
					\partial\mathrm P_X^1
					\otimes
					\bar\partial\overline{\mathrm P_X^1}
				}
				=0.
			\end{align*}
			Hence $\iota_{\mathrm{BC},\partial}^{2,2}$ is injective,
			and $Y$ satisfies
			$
			\mathbb B^{2,2}.$

			We shall also need two further injectivities in the next
			example. Since $\partial\mathrm P_X^2=0$, for every $q$ the
			map
			\begin{equation*}
				\iota_{\mathrm{BC},\partial}^{3,q}:
				H_{\mathrm{BC}}^{3,q}(X)
				\longrightarrow
				H_{\partial}^{3,q}(X)
			\end{equation*}
			is identified with the inclusion
			\begin{equation*}
				\mathrm P_X^3
				\otimes
				\ker\left(
				\bar\partial:
				\overline{\mathrm P_X^q}
				\longrightarrow
				\overline{\mathrm P_X^{q+1}}
				\right)
				\hookrightarrow
				\mathrm P_X^3
				\otimes
				\overline{\mathrm P_X^q}.
			\end{equation*}
			In particular, $X$ also satisfies
			$
			\mathbb B^{3,0}
			$ and $
			\mathbb B^{3,1}.
			$

					\medskip
			\item
			\textit{We now show that \(
				E_1^{\bullet,\bullet}(X)
				\not\simeq
				E_\infty^{\bullet,\bullet}(X)
				\).}
			Notice that
			\begin{equation*}
				\bar\partial\omega_1=0,
				\qquad
				d_1[\omega_1]
				=
				[\partial\omega_1]
				=
				[\omega_2\wedge\omega_3]
				\neq0.
			\end{equation*}
			The last class is nonzero because a nonzero form of type
			$(2,0)$ cannot be $\bar\partial$-exact. Thus the Fr\"olicher
			spectral sequence of $X$ does not degenerate at $E_1$.
			
					\medskip
			\item
			\textit{We finally verify \ref{4444}.}  By
			\eqref{eq:sl2-small-cohomology} and complex conjugation, 
			$H^2(X,\mathcal O_X)=0. $  Therefore
			$H^2(X,T_X^{1,0})=0$
			since $X$ is complex parallelisable. 
			Let $E\rightarrow X$ be a holomorphic vector bundle satisfying
			$H^2(X,\operatorname{End}^{0}E)=0$. By \eqref{direct sum}, 
			$H^2(X,\operatorname{End}E)=0.$
			The Atiyah sequence
			\begin{equation*}
				0\longrightarrow
				\operatorname{End}E
				\longrightarrow
				A(E)
				\longrightarrow
				T_X^{1,0}
				\longrightarrow0
			\end{equation*}
			then gives
			\begin{equation*}
				H^2(X,A(E))=0.
			\end{equation*}
		\end{enumerate}
	\end{proof}

	\begin{proposition}\label{prop:nonzero-obstruction-example}
		There exist a compact complex sixfold $X$ and a simple
		holomorphic vector bundle $E$ of rank $2$ over $X$ such that:
		\begin{enumerate}[{(1)}]
			\item $K_X\simeq\mathcal O_X$;
			
			\item $X$ satisfies $\mathbb S^{6,1}$,
			$\mathbb S^{6,2}$ and $\mathbb B^{5,2}$;
			
			\item the Fr\"olicher spectral sequence of $X$ does not
			degenerate at $E_1$; in particular, $X$ does not satisfy
			the full $\partial\bar\partial$-lemma;
			
			\item the bundle $E$ satisfies
			\begin{equation*}
				H^2(X,\operatorname{End}^{0}E)=0,
				\quad
				H^2(X,\operatorname{End}E)\simeq\mathbb C,
				\quad
				H^2(X,A(E))\simeq\mathbb C^2.
			\end{equation*}
		\end{enumerate}
		Moreover, the forgetful morphism
		$\operatorname{Def}(X,E)\longrightarrow\operatorname{Def}(X)$
		is not smooth.
	\end{proposition}
	
	\begin{proof}
		Let
		\begin{equation*}
			\zeta:=e^{\pi i/3},
			\quad
			R:=\mathbb Z[\zeta],
			\quad
			B:=\mathbb C/R.
		\end{equation*}
		Consider the complex Heisenberg group $H=\mathbb C^3$, with
		multiplication
		\begin{equation*}
			(z_1,z_2,z_3)(w_1,w_2,w_3)
			=
			(z_1+w_1,z_2+w_2,z_3+w_3+z_1w_2),
		\end{equation*}
		and its cocompact lattice $\Gamma:=R^3$.
		Set $N:=\Gamma\backslash H$.
		The automorphism
		\begin{equation*}
			\phi(z_1,z_2,z_3)
			=
			(\zeta z_1,\zeta z_2,\zeta^2z_3)
		\end{equation*}
		preserves $\Gamma$ and therefore descends to $N$.

				\medskip
		On $\widetilde X:=N\times B^3$, define
		\begin{equation*}
			g(z,x,y,t)
			=
			\left(\phi(z),\zeta^2x,y+\frac16,t\right).
		\end{equation*}
		This generates a free action of a cyclic group of order $6$:
		for $1\leq k\leq5$, the map $g^k$ translates the
		$y$-coordinate by the nonzero element $\frac k6$ of $B$.
		Thus
		\begin{equation*}
			X:=\widetilde X/\langle g\rangle
		\end{equation*}
		is a compact complex sixfold.
		Since $g$ acts trivially on the last factor,
		$X=Z\times B$, where $Z$ is the quotient of $N\times B^2$
		by the induced action. Denote the projection to the last
		factor by $$f:X\longrightarrow B.$$
		
		On $\widetilde X$, consider the holomorphic coframe
		\begin{equation*}
			\alpha_1=dz_1,\quad
			\alpha_2=dz_2,\quad
			\alpha_3=dz_3-z_1dz_2,\quad
			\alpha_4=dx,\quad
			\alpha_5=dy,\quad
			\alpha_6=dt.
		\end{equation*}
		Writing
		$\alpha_{j_1\cdots j_k}
		:=\alpha_{j_1}\wedge\cdots\wedge\alpha_{j_k}$,
		we have
		\begin{equation}\label{eq:nilpotent-structure}
			\partial\alpha_3=-\alpha_{12},
			\quad
			\partial\alpha_j=0
			\quad(j\neq3),
		\end{equation}
		and
		\begin{equation}\label{eq:cyclic-action}
			g^*(\alpha_1,\ldots,\alpha_6)
			=
			(\zeta\alpha_1,\zeta\alpha_2,
			\zeta^2\alpha_3,\zeta^2\alpha_4,
			\alpha_5,\alpha_6).
		\end{equation}
		We identify $g$-invariant forms on $\widetilde X$ with
		their descended forms on $X$.
		
		\begin{enumerate}[{(1)}]
			\item
			\textit{First, $K_X$ is trivial.}
			The nowhere-vanishing holomorphic form
			\begin{equation*}
				\Omega:=\alpha_{123456}
			\end{equation*}
			is $g$-invariant, since $\zeta^{1+1+2+2}=1$.
			It therefore descends to $X$, proving
			$K_X\simeq\mathcal O_X$.

					\medskip
			\item
			\textit{We next verify the three weak
				$\partial\bar\partial$-conditions.}
			Put
			\begin{equation*}
				\mathrm P^\bullet
				:=
				\Lambda^\bullet
				\langle\alpha_1,\ldots,\alpha_6\rangle,
				\quad
				\mathcal I^{p,q}
				:=
				\left(
				\mathrm P^p\otimes\overline{\mathrm P^q}
				\right)^{\langle g\rangle}.
			\end{equation*}
			By Sakane's theorem \cite[Theorem  1]{Sak76},
			left-invariant forms compute the Dolbeault cohomology
			of $\widetilde X$. Since taking invariants under a
			finite group is exact, the inclusion
			\begin{equation*}
				\mathcal I^{\bullet,\bullet}
				\hookrightarrow
				A^{\bullet,\bullet}(X)
			\end{equation*}
			also induces isomorphisms in Dolbeault cohomology.
			By complex conjugation and \cite[Corollary  5]{KS23},
			it induces isomorphisms in $\partial$-, Bott--Chern
			and Aeppli cohomology as well.
			Thus the natural maps characterizing the weak
			conditions may be computed on
			$\mathcal I^{\bullet,\bullet}$.
			
					\medskip
			By \eqref{eq:cyclic-action},
			\begin{equation*}
				\begin{aligned}
					\mathcal I^{6,1}
					&=
					\left\langle
					\Omega\wedge\bar\alpha_5,\,
					\Omega\wedge\bar\alpha_6
					\right\rangle,
					\\
					\mathcal I^{6,2}
					&=
					\mathbb C\,
					\Omega\wedge\bar\alpha_{56}.
				\end{aligned}
			\end{equation*}
			These forms are $\bar\partial$-closed. Moreover,
			\begin{equation*}
				\partial\mathrm P^5=0,
				\quad
				\bar\partial\mathcal I^{6,0}
				=
				\bar\partial\mathcal I^{6,1}
				=
				0.
			\end{equation*}
			Consequently, for $q=1,2$, both
			$H_{\bar\partial}^{6,q}(X)$ and
			$H_{\mathrm A}^{6,q}(X)$ are identified with
			$\mathcal I^{6,q}$, and the natural map between them
			is the identity. Hence $X$ satisfies
			$\mathbb S^{6,1}$ and $\mathbb S^{6,2}$.
			
					\medskip
			It remains to check $\mathbb B^{5,2}$.
			The same action gives
			\begin{equation*}
				\begin{aligned}
					\mathcal I^{5,2}
					=
					\bigl\langle\,
					&
					\alpha_{12345}\wedge\bar\alpha_{56},\,
					\alpha_{12346}\wedge\bar\alpha_{56},
					\\
					&
					\alpha_{12356}\wedge\bar\alpha_{34},\,
					\alpha_{12456}\wedge\bar\alpha_{34}
					\,\bigr\rangle.
				\end{aligned}
			\end{equation*}
			All four generators are $\partial$-closed, while
			the $\bar\partial$-closed subspace is spanned by
			the first two, since
			$\bar\partial\bar\alpha_{34}=-\bar\alpha_{124}$.
			Furthermore,
			\begin{equation*}
				\partial\mathcal I^{4,2}
				=
				\mathbb C\,
				\alpha_{12456}\wedge\bar\alpha_{34},
			\end{equation*}
			which contains no nonzero $\bar\partial$-closed
			element. Thus, in the invariant double complex
			$\mathcal I^{\bullet,\bullet}$, there are no nonzero
			$\partial\bar\partial$-boundaries in bidegree $(5,2)$.
			We obtain
			\begin{equation*}
				\begin{aligned}
					H_{\mathrm{BC}}^{5,2}(X)
					&=
					\left\langle
					[\alpha_{12345}\wedge\bar\alpha_{56}],
					[\alpha_{12346}\wedge\bar\alpha_{56}]
					\right\rangle,
					\\
					H_\partial^{5,2}(X)
					&=
					\left\langle
					[\alpha_{12345}\wedge\bar\alpha_{56}],
					[\alpha_{12346}\wedge\bar\alpha_{56}],
					[\alpha_{12356}\wedge\bar\alpha_{34}]
					\right\rangle.
				\end{aligned}
			\end{equation*}
			The natural map is the inclusion of the first two
			classes. Therefore $X$ satisfies $\mathbb B^{5,2}$.

					\medskip
			\item
			\textit{We now show that the Fr\"olicher spectral
				sequence does not degenerate at $E_1$.}
			The form $\alpha_3\wedge\bar\alpha_4$ is
			$g$-invariant and $\bar\partial$-closed, and
			\begin{equation*}
				d_1[\alpha_3\wedge\bar\alpha_4]
				=
				-[\alpha_{12}\wedge\bar\alpha_4]
				\neq0.
			\end{equation*}
			Indeed, the last class is already nonzero on
			$\widetilde X$, whose Dolbeault cohomology in this
			bidegree is
			$\mathrm P^2\otimes
			H^1(\overline{\mathrm P^\bullet},\bar\partial)$,
			and $[\bar\alpha_4]\neq0$.
			In particular, $X$ does not satisfy the full
			$\partial\bar\partial$-lemma.

					\medskip
			\item
			\textit{Finally, we construct $E$ and compute its
				obstruction space.}
			Recall that a holomorphic vector bundle $V$ over $B$
			is \textit{stable} if
			$\frac{\deg S}{\operatorname{rank}S}
			<
			\frac{\deg V}{\operatorname{rank}V}$
			for every nonzero proper holomorphic subbundle
			$S\subset V$.
			
			Let $o:=[0]\in B$ be the origin and choose a nonsplit
			extension
			\begin{equation*}
				0
				\longrightarrow
				\mathcal O_B
				\longrightarrow
				F
				\longrightarrow
				\mathcal O_B(o)
				\longrightarrow
				0.
			\end{equation*}
			Such an extension exists since
			$H^1(B,\mathcal O_B(-o))\simeq\mathbb C$.
			Then $F$ has rank $2$ and degree $1$.
			If a line subbundle $S\subset F$ had $\deg S\geq1$,
			the composite $S\to F\to\mathcal O_B(o)$ would be
			nonzero, since
			$\operatorname{Hom}(S,\mathcal O_B)=0$.
			As $\deg\mathcal O_B(o)=1$, this composite would be
			an isomorphism, splitting the extension.
			Thus every line subbundle has degree at most $0$,
			so $F$ is stable and hence simple by
			\cite[Chapter  V, Corollary  7.14]{Kob87}.
			Since $K_B\simeq\mathcal O_B$, Serre duality gives
			\begin{equation*}
				H^0(B,\operatorname{End}^{0}F)
				=
				H^1(B,\operatorname{End}^{0}F)
				=
				0.
			\end{equation*}
			
			Set $$E:=f^*F.$$
			As $X=Z\times B$, the K\"unneth formula yields
			\begin{equation*}
				H^q(X,\operatorname{End}^{0}E)=0
				\quad\text{for every }q.
			\end{equation*}
			In particular, $E$ is simple.

					\medskip
			We record the cohomology groups needed below.
			A direct calculation from
			\eqref{eq:nilpotent-structure} gives
			\begin{equation*}
				\begin{aligned}
					H^2(\mathrm P^\bullet,\partial)
					={}&
					\mathbb C[\alpha_{56}]
					\\
					&\oplus
					\left\langle
					[\alpha_{15}],[\alpha_{16}],
					[\alpha_{25}],[\alpha_{26}]
					\right\rangle
					\\
					&\oplus
					\left\langle
					[\alpha_{45}],[\alpha_{46}]
					\right\rangle
					\\
					&\oplus
					\left\langle
					[\alpha_{13}],[\alpha_{14}],
					[\alpha_{23}],[\alpha_{24}]
					\right\rangle,
				\end{aligned}
			\end{equation*}
			where the four summands have eigenvalues
			$1,\zeta,\zeta^2,\zeta^3$, respectively.
			In degree $3$, all eigenvalues belong to
			$\{\zeta,\zeta^2,\zeta^3,\zeta^4,\zeta^5\}$.
			Taking complex conjugates and invariants therefore gives
			\begin{equation}\label{eq:quotient-O-cohomology}
				H^2(X,\mathcal O_X)
				=
				\mathbb C[\bar\alpha_{56}],
				\quad
				H^3(X,\mathcal O_X)=0.
			\end{equation}
			By \eqref{direct sum}, \eqref{eq:quotient-O-cohomology}
			and the vanishing of $H^q(X,\operatorname{End}^{0}E)$,
			the trace map gives
			\begin{equation*}
				H^2(X,\operatorname{End}E)
				\simeq
				H^2(X,\mathcal O_X)
				\simeq\mathbb C,
				\quad
				H^3(X,\operatorname{End}E)=0.
			\end{equation*}
			
			Let $v_1,\ldots,v_6$ be the holomorphic vector fields
			on $\widetilde X$ dual to
			$\alpha_1,\ldots,\alpha_6$.
			Their eigenvalues under pull-back are
			$\zeta^{-1},\zeta^{-1},
			\zeta^{-2},\zeta^{-2},1,1$.
			The preceding calculation consequently gives
			\begin{equation*}
				H^2(X,T_X^{1,0})
				=
				\left\langle
				[\bar\alpha_{56}\otimes v_5],\,
				[\bar\alpha_{56}\otimes v_6]
				\right\rangle
				\simeq
				\mathbb C^2.
			\end{equation*}
			
			Consider the connecting map 
			\begin{equation}\label{conne}
				H^1(X,T_X^{1,0})
				\xrightarrow{\;\lrcorner\mathrm{At}(E)\;}
				H^2(X,\operatorname{End}E).
			\end{equation}
			Since $E=f^*F$, we have
			$\operatorname{tr}\mathrm{At}(E)
			=f^*\operatorname{tr}\mathrm{At}(F)$.
			Moreover, $\operatorname{tr}\mathrm{At}(F)$ is a
			nonzero scalar multiple of $c_1(F)$, which is nonzero
			because $\deg F=1$.
			Contracting with $v_6$ and wedging with $\bar\alpha_5$
			therefore gives
			\begin{equation*}
				\operatorname{tr}\left(
				[\bar\alpha_5\otimes v_6]
				\lrcorner\mathrm{At}(E)
				\right)
				=
				c[\bar\alpha_{56}],
				\quad c\neq0.
			\end{equation*}
			Thus the connecting map \eqref{conne} is nonzero and hence
			surjective, since its target is one-dimensional.
			Together with $H^3(X,\operatorname{End}E)=0$,
			the long exact sequence associated with the Atiyah
			sequence gives
			\begin{equation*}
				H^2(X,A(E))
				\simeq
				H^2(X,T_X^{1,0})
				\simeq
				\mathbb C^2.
			\end{equation*}
			
					\medskip
			Finally, the same exact sequence shows that the tangent
			map of the forgetful morphism has image
			\begin{equation*}
				\operatorname{im}\left(
				H^1(X,A(E))
				\longrightarrow H^1(X,T_X^{1,0})
				\right)
				=
				\ker\bigl(\lrcorner\mathrm{At}(E)\bigr).
			\end{equation*}
			Since \eqref{conne} is nonzero, this
			tangent map is not surjective.
			Hence the forgetful morphism $\operatorname{Def}(X,E)\longrightarrow\operatorname{Def}(X)$  is not smooth by
			\cite[Corollary  3.6.10]{Man22}.
		\end{enumerate}
	\end{proof}
	
	\begin{remark}\label{rem:torsion-canonical}
		It is natural to ask whether Theorem  \ref{main thm2} remains
		true when $K_X$ is merely \emph{torsion}. A natural strategy,
		whether one follows our Tian--Todorov approach or the
		$T^1$-lifting approach of Iacono--Manetti, is to pass to the
		canonical cyclic \'etale cover
		$\pi:\widetilde X\rightarrow X$ as in \eqref{diagram},
		apply Theorem  \ref{main thm2} to
		$(\widetilde X,\pi^*(\det E))$, and then use
		Lemma  \ref{wells} to descend the vanishing of the obstruction
		classes.
		
				\medskip
		However, even when $X$ is a $\p\bp$-manifold, it is not clear
		whether $\widetilde X$ is again a $\p\bp$-manifold. Notice
		that the converse direction holds in much greater
		generality. Indeed, for any surjective morphism
		$f:Y\rightarrow X$ between two equidimensional compact
		complex manifolds, if $Y$ is a $\p\bp$-manifold, then so is
		$X$; see, e.g.,
		\cite[Theorem  1.1 \& Example  3.1(3)]{Men22}.
		
		This leads to the following:
		\begin{question}
			Let $X$ be an $n$-dimensional compact complex manifold
			with torsion canonical bundle. If $X$ is a
			$\p\bp$-manifold
			\emph{(resp.\ if $X$ satisfies $\mathbb S^{n,1}$,
				$\mathbb S^{n,2}$ and $\mathbb B^{n-1,2}$)},
			must its canonical cyclic \'etale cover $\widetilde X$
			again be a $\p\bp$-manifold
			\emph{(resp.\ satisfy $\mathbb S^{n,1}$,
				$\mathbb S^{n,2}$ and $\mathbb B^{n-1,2}$)},
			or can one construct a counterexample to either
			assertion?
		\end{question}
	\end{remark}

	\section{Simultaneous deformations in Thomas's example}
	
	In this section, we recall Thomas's example \cite{Tho99}
	and Felten's use of it \cite{Fel25}.
	We then show that the deformation functor $\operatorname{Def}(X,E)$ is smooth,
	although the forgetful map to $\operatorname{Def}(X)$ is not.
	
	\medskip
	\medskip
	Let
	$$
	P:=\mathbb P^2_{[x:y:z]}\times\mathbb P^2_{[u:v:w]}.
	$$
	Following \cite[\S 2]{Tho99}, choose a sufficiently small
	general polynomial $p$ of bidegree $(3,3)$ vanishing
	scheme-theoretically on 
	$$
	f:=\{u=v=0,\ x^2y=0\}\subset P.
	$$
	Set
	$$
	F=x^2yw^3+z^2yvw^2+z^3uw^2+(x+y)^3vw^2+p.
	$$
	Then $X:=\{F=0\}\subset P$ is a smooth threefold containing
	the line
	$$
	C:=\{x=u=v=0\}\cong\mathbb P^1.
	$$
	The choice of $p$ preserves the fibre $f$, in which $C$
	occurs with multiplicity two.

	\medskip
	We construct the bundle following \cite[\S 3]{Tho99}.
	Starting from $C\subset X$, choose
	$L=\mathcal O_X(-2,2)$.\footnote{Write
		$\mathcal O_P(i,j)
		:=\operatorname{pr}_1^*\mathcal O_{\mathbb P^2}(i)
		\otimes\operatorname{pr}_2^*\mathcal O_{\mathbb P^2}(j)$
		for the two projections of $P$, and set
		$\mathcal O_X(i,j):=\mathcal O_P(i,j)|_X$.}
	Writing $N_{C/X}$ for the normal bundle, Thomas verified
	the two conditions needed for the Serre construction:
	$$
	\det N_{C/X}\cong\mathcal O_C(-2)\cong L|_C,
	\quad
	H^2(X,L^{-1})=0.
	$$
	This gives a rank-two bundle $V$ with $\det V\cong L$
	and a section $s\in H^0(X,V)$ whose zero locus is $C$.
	If $\mathcal I_C$ denotes the ideal sheaf of $C$ in $X$,
	the resulting exact sequence is
	\begin{equation}\label{thomas-serre}
		0\longrightarrow\mathcal O_X
		\xrightarrow{s}V
		\longrightarrow\mathcal I_C\otimes L
		\longrightarrow0.
	\end{equation}
	Set
	\begin{equation}\label{EEE}
		E:=V\otimes\mathcal O_X(1,-1).
	\end{equation}
	Since $V$ has rank two,
	$$
	\det E
	\cong L\otimes\mathcal O_X(2,-2)
	\cong\mathcal O_X.
	$$

	\medskip
	As Yukinobu Toda observed,\footnote{See the Addendum to the revised version of
		\cite{Tho99}, available at
		\href{https://arxiv.org/abs/math/9903034v2}{arXiv:math/9903034v2}.}
	deforming $E$ on the \textit{fixed} threefold $X$ is equivalent
	to deforming $C$ inside $X$:
	\begin{equation}\label{thomas-fixed-correspondence}
		\operatorname{Def}_X(E)\cong\operatorname{Def}_X(C).
	\end{equation}
	Thomas computed the local deformation ring
	and hence the corresponding Kuranishi space:
	$$
	R_0=\mathbb C[\epsilon,\eta]/(\epsilon^2,\eta^2),
	\quad
	\operatorname{Kur}_X(E)\cong\operatorname{Spec}R_0.
	$$
	A first-order deformation sends
	$\epsilon\mapsto at$ and $\eta\mapsto bt$
	in $\mathbb C[t]/(t^2)$.
	Any lift to $\mathbb C[t]/(t^3)$ must send them to
	$at+a_2t^2$ and $bt+b_2t^2$.
	The relations $\epsilon^2=\eta^2=0$ then force
	$a^2=b^2=0$.
	Thus every nontrivial first-order deformation of $E$
	with $X$ fixed is obstructed at second order,
	as stated in \cite[Theorem 1.1]{Tho99}.

	\begin{remark}\label{thomas} By  	\cite[Equation (3.2) with $k=l=0$]{Tho99}  plus the K\"unneth formula
		\begin{equation}\label{1111111}
			H^1(X,\mathcal{O}_X)=H^2(X,\mathcal{O}_X)=0.
		\end{equation}
		Therefore the trace decomposition
		\eqref{direct sum} gives
		$H^2(X,\End^0E)\simeq H^2(X,\End E)\neq0$,
		as the deformations of $E$ with $X$ fixed are obstructed.
	\end{remark}

	\medskip
	\medskip
	\medskip
	In
	\cite[Example 11.14]{Fel25}, Felten considered the trivial first- and second-order deformations:
	$$
	X_1^{\mathrm{triv}}
	:=X\times\operatorname{Spec}\mathbb C[t]/(t^2),
	\quad
	X_2^{\mathrm{triv}}
	:=X\times\operatorname{Spec}\mathbb C[t]/(t^3).
	$$
	Then by \cite[Theorem 1.1]{Tho99} a nontrivial deformation $E_1$ of $E$ from \eqref{EEE} on
	$X_1^{\mathrm{triv}}$ cannot extend to a bundle on
	$X_2^{\mathrm{triv}}$.
	Consequently, the forgetful map
	$$
	\pi:\operatorname{Def}(X,E)\longrightarrow\operatorname{Def}(X)
	$$
	is not smooth.

	\medskip
	\medskip
	For simultaneous deformations, however, the lifted
	threefold is \textit{not prescribed}.
	We will prove that $\operatorname{Def}(X,E)$ is smooth
	by lifting arbitrary deformations of the pair across
	arbitrary Artin extensions 
	(in the example above, we may therefore replace
	$X_2^{\mathrm{triv}}$ by another lift while keeping
	$(X_1^{\mathrm{triv}},E_1)$ unchanged).
	The idea is to extend \eqref{thomas-fixed-correspondence}
	to simultaneous deformations, thereby reducing the problem
	to lifting the curve together with the threefold.
	We then lift the curve in $P$ and correct the bicubic
	equation to contain it without changing the given deformation.
	The Serre construction recovers the lifted bundle.
	
	\begin{theorem}\label{Thomas}
		The deformation functor $\operatorname{Def}(X,E)$ is smooth.
	\end{theorem}
	\begin{proof}
		Let $R'\twoheadrightarrow R$ be a surjection of local Artin
		$\mathbb C$-algebras with kernel $J$, and let $(X_R,E_R)$
		be a given deformation of $(X,E)$ over $R$.
		We will construct a lift of this pair over $R'$.
		The proof is divided into two steps.
		
		\setcounter{step}{0}
		
		\begin{step}{Extend the bundle--curve correspondence.}
			\label{step:thomas-curve}
		\end{step}
		
		We first show that \eqref{thomas-fixed-correspondence}
		also holds when $X$ varies:
		\begin{equation}\label{thomas-correspondence}
			\operatorname{Def}(X,E)\cong\operatorname{Def}(X,C).
		\end{equation}

 Thomas proved the following cohomological identities:
		\begin{equation}\label{thomas-vanishing}
			\begin{gathered}
				h^0(X,V)=1,\quad H^1(X,V)=0,\\
				H^1(X,L^{-1})=H^2(X,L^{-1})=0.
			\end{gathered}
		\end{equation}
		
By \eqref{1111111}, the exact sequence in
\cite[Exercise~III. 4.6]{Har77}, applied successively
along small Artin extensions, gives
		\begin{equation}\label{thomas-picard}
			\operatorname{Pic}(X_R)\xrightarrow{\sim}
			\operatorname{Pic}(X).
		\end{equation}
		Thus $\mathcal O_X(-1,1)$ and $L$ extend uniquely up to
		isomorphism to $\mathcal O_{X_R}(-1,1)$ and $L_R$.
		By \eqref{EEE}, set
		$V_R:=E_R\otimes\mathcal O_{X_R}(-1,1)$.
		Since $\det V_R$ restricts to $L$,
		\eqref{thomas-picard} gives $\det V_R\cong L_R$.

		\medskip
		The equality $H^1(X,V)=0$ allows $s$ to extend to $s_R$.
		Together with $H^0(X,V)=\mathbb C s$, this gives
		$$
		H^0(X_R,V_R)=R\,s_R,\quad s_R|_X=s.
		$$
		Any other extension of $s$ is therefore $u s_R$
		for an invertible element $u\in R$.
		Since $Z(u s_R)=Z(s_R)$, the closed subspace
		$C_R:=Z(s_R)$ is independent of this choice.
		Moreover, $s$ cuts out $C$ in codimension two,
		so $C_R$ is a flat deformation of $C$.
		Hence the given bundle $E_R$ determines $C_R\subset X_R$.
		
				\medskip
		Conversely, suppose that $C_R\subset X_R$ is given.
		Both $\det N_{C_R/X_R}$ and $L_R|_{C_R}$ extend $L|_C$.
		Since $H^1(C,\mathcal O_C)=0$, the same uniqueness
		argument as in \eqref{thomas-picard} gives
		$$
		\det N_{C_R/X_R}\cong L_R|_{C_R}.
		$$
		The vanishings for $L^{-1}$ in \eqref{thomas-vanishing}
		also hold for $L_R^{-1}$.
		The Serre construction of \cite[\S 3]{Tho99}
		therefore applies over $R$:
		$H^2(X_R,L_R^{-1})=0$ gives a rank-two bundle $V_R$
		in the sequence \eqref{thomas-serre} over $R$, and
		$H^1(X_R,L_R^{-1})=0$ makes it unique up to isomorphism.
		Twisting by $\mathcal O_{X_R}(1,-1)$ as in \eqref{EEE}
		recovers $E_R$. This proves \eqref{thomas-correspondence}.
		
		\begin{step}{Lift the curve and the threefold.}
			\label{step:thomas-lift}
		\end{step}
		
		By \eqref{thomas-correspondence}, it remains to lift
		$(X_R,C_R)$.
		
		By \eqref{thomas-picard}, the line bundles
		$\mathcal O_X(1,0)$ and $\mathcal O_X(0,1)$ extend to $X_R$.
		Moreover, \cite[Equation (3.2) with
		$(k,l)=(1,0)$ \& $(0,1)$]{Tho99}, gives
		$$
		H^1(X,\mathcal O_X(1,0))
		=
		H^1(X,\mathcal O_X(0,1))
		=0.
		$$
		Thus the coordinate sections $x,y,z$ and $u,v,w$
		extend and still generate these line bundles.
		They define a morphism
		$$
		X_R\longrightarrow P_R:=P\times\operatorname{Spec}R
		$$
		extending $X\subset P$,  which remains a closed
		immersion by Nakayama's lemma.
		Since $X_R$ is flat over $R$ and $X\subset P$ is
		a hypersurface, the immersion $X_R\subset P_R$ is
		regular of codimension one
		\cite[Proposition~D.2.5]{Ser06}.
		Its ideal is therefore a line bundle extending
		$\mathcal O_P(-3,-3)$. 
		As $H^1(P,\mathcal O_P)=0$, the same uniqueness argument
		as in \eqref{thomas-picard} identifies this ideal with
		$\mathcal O_{P_R}(-3,-3)$.
		Hence
		$$
		X_R=\{F_R=0\}\subset P_R
		$$
		for a bicubic polynomial $F_R$ lifting $F$.
		
		\medskip
		In the Addendum to the revised version of \cite{Tho99},
		Thomas described the deformations of $C$ inside $P$
		as a moving line in the first factor together with
		a moving point in the second.
		Writing $\mathfrak m_R$ for the maximal ideal of $R$,
		his description gives
		\begin{equation}\label{thomas-embedded-pair}
			C_R=\{x=ay+bz,\ u=cw,\ v=dw\}
			\subset\{F_R=0\},
			\quad a,b,c,d\in\mathfrak m_R.
		\end{equation}

				\medskip
		Choose lifts $\widetilde a,\widetilde b,\widetilde c,
		\widetilde d$ to $R'$.
		The same equations define
		$$
		C_{R'}:=
		\{x=\widetilde a y+\widetilde b z,\
		u=\widetilde c w,\
		v=\widetilde d w\}
		\cong\mathbb P^1_{R'}
		$$
		inside $P_{R'}:=P\times\operatorname{Spec}R'$.
		Thus $C_{R'}$ is a flat lift of $C_R$.
		Choose also a bicubic lift $\widetilde F$ of $F_R$ and set
		$$
		Q(y,z):=
		\widetilde F
		\bigl(\widetilde a y+\widetilde b z,\ y,\ z;\
		\widetilde c,\ \widetilde d,\ 1\bigr).
		$$
		The containment in \eqref{thomas-embedded-pair} implies
		that every coefficient of the homogeneous cubic $Q$
		lies in $J$.
		Since $\widetilde F$ has degree three in $u,v,w$,
		$$
		\widetilde F|_{C_{R'}}=w^3Q(y,z).
		$$
		We therefore set
		\begin{equation}\label{thomas-correction}
			F_{R'}:=\widetilde F-w^3Q(y,z).
		\end{equation}
		The correction is again bicubic, and
		\eqref{thomas-correction} gives
		\begin{equation}\label{thomas-lifting}
			F_{R'}\bmod J=F_R,
			\quad
			F_{R'}|_{C_{R'}}=0.
		\end{equation}
		
				\medskip
		Put $X_{R'}:=\{F_{R'}=0\}\subset P_{R'}$.
		Since its central fibre $X$ is a smooth hypersurface,
		the Jacobian criterion shows that $X_{R'}$ is smooth,
		hence flat, over $R'$.
		Thus \eqref{thomas-lifting} gives a lift
		$(X_{R'},C_{R'})$ of $(X_R,C_R)$.
		By \eqref{thomas-correspondence}, this lifted pair
		determines a bundle $E_{R'}$ restricting to $E_R$.
		This proves that $\operatorname{Def}(X,E)$ is smooth.
	\end{proof}

\section{Simple bundles with obstructed joint deformations}

\begin{proof}[Proof of Theorem~\ref{thm:simple-obstructed-pairs}]
	The proof is divided into four steps.
	\setcounter{step}{0}
	
	\begin{step}{An obstructed deformation of an ideal sheaf.}
	\end{step}
	
	Fix $p\in X$ and local holomorphic coordinates
	$z_1,\ldots,z_n$ centred at $p$.
	Let $Z\subset X$ be the zero-dimensional closed
	analytic subspace supported at $p$ with ideal
	$$
	\mathcal I_{Z,p}
	=
	(z_4,\ldots,z_n)+(z_1,z_2,z_3)^2,
	$$
	where the first summand is omitted when $n=3$.
	Then $\operatorname{length}(Z)=4$.
	
	\medskip
	For $n=3$, \cite[Theorem~2.1]{HMS22} identifies
	the germ of the Hilbert scheme at $[Z]$ with
	$(\mathbb C^3,0)\times(C\operatorname{Gr}(2,6),0)$.
	For $n>3$, on the chart with quotient basis
	$1,z_1,z_2,z_3$, each additional coordinate
	contributes four free parameters. Hence
	\begin{equation}\label{eq:simple-hilbert-model}
		\bigl(\operatorname{Hilb}^4(X),[Z]\bigr)
		\simeq
		(\mathbb C^{4n-9},0)\times
		\bigl(C\operatorname{Gr}(2,6),0\bigr).
	\end{equation}
	Here $C\operatorname{Gr}(2,6)\subset\mathbb C^{15}$
	is the affine cone over the Pl\"ucker embedding,
	with coordinates $p_{ij}$, $1\leq i<j\leq6$,
	and quadratic defining relations
	\cite[\S1.3]{HMS22}.
	Choose $v$ with $p_{12}(v)=p_{34}(v)=1$
	and all other coordinates zero.
	Since the relations are quadratic, $\epsilon v$
	defines a first-order \textit{embedded deformation}  of $Z$,
	namely a flat deformation as a closed analytic
	subspace of the fixed manifold $X$;
	see \cite[\S3.2.1]{Ser06}.
	
		\medskip
	The ideal sequence
	$$
	0\longrightarrow\mathcal I_Z
	\longrightarrow\mathcal O_X
	\longrightarrow\mathcal O_Z
	\longrightarrow0
	$$
	induces
	$$
	H^1(X,\mathcal O_X)
	\longrightarrow
	\operatorname{Ext}^1_X(\mathcal I_Z,\mathcal O_X)
	\longrightarrow
	\operatorname{Ext}^2_X(\mathcal O_Z,\mathcal O_X),
	$$
	where $\operatorname{Ext}^q_X$ denotes the global
	extension groups of coherent sheaves.
	The first term vanishes by assumption, while
	Serre duality gives
	$$
	\operatorname{Ext}^2_X(\mathcal O_Z,\mathcal O_X)
	\simeq H^{n-2}(X,\mathcal O_Z)^*=0.
	$$
	Thus
	\begin{equation}\label{eq:simple-ideal-ext-vanishing}
		\operatorname{Ext}^1_X(\mathcal I_Z,\mathcal O_X)=0.
	\end{equation}
	By \eqref{eq:simple-ideal-ext-vanishing} and
	\cite[Corollary~A.6]{Ser06}, a second-order lift
	of the corresponding ideal sheaf would yield
	an embedded lift.
	However, the Pl\"ucker relation
	$$
	q=p_{12}p_{34}-p_{13}p_{24}+p_{14}p_{23}
	$$
	would then require
	$$
	0=q(\epsilon v+\epsilon^2w)
	=\epsilon^2\neq0
	\quad\bmod\epsilon^3,
	$$
	a contradiction.
	Consequently, the corresponding first-order class
	$\alpha\in\operatorname{Ext}^1_X(\mathcal I_Z,\mathcal I_Z)$
	has nonzero quadratic obstruction
	\cite[\S2.3]{CPZ24}:
	\begin{equation}\label{eq:simple-ideal-obstruction}
		\alpha\circ\alpha\neq0
		\quad\text{in}\quad
		\operatorname{Ext}^2_X(\mathcal I_Z,\mathcal I_Z).
	\end{equation}

		\medskip
	Recall that a coherent sheaf $\mathcal F$ is
	\textit{simple} if
	$\operatorname{Hom}_X(\mathcal F,\mathcal F)
	=\mathbb C\cdot\operatorname{id}_{\mathcal F}$.
	An endomorphism of $\mathcal I_Z$ restricts on
	$X\setminus\{p\}$ to multiplication by a holomorphic
	function, which extends to a constant on $X$
	by Hartogs' theorem and compactness.
	Since $\mathcal I_Z$ is torsion-free, the endomorphism
	is scalar on all of $X$. Thus $\mathcal I_Z$ is simple.
	
	\medskip
	\begin{step}{Construction of the vector bundle.}
	\end{step}
	
	We follow the construction of Joyce--Song in \cite[\S8]{JS12}.
	Fix a very ample line bundle $\mathcal O_X(1)$
	and set $G_0:=\mathcal I_Z$.
	By \cite[Theorems~II.5.17 \& III.5.2]{Har77},
	we may successively choose $m_j\gg0$ such that
	$G_{j-1}(m_j)$ is globally generated and
	\begin{equation}\label{eq:simple-serre-vanishing}
		H^q\bigl(X,G_{j-1}(m_j)\bigr)=0
		\qquad(q>0).
	\end{equation}
	Put $V_j:=H^0(X,G_{j-1}(m_j))$ and consider
	\begin{equation}\label{eq:simple-evaluation}
		0\longrightarrow G_j
		\longrightarrow V_j\otimes\mathcal O_X(-m_j)
		\longrightarrow G_{j-1}
		\longrightarrow0,
		\qquad 1\leq j\leq n-1.
	\end{equation}
	
		\medskip
	Since $X$ is smooth of dimension $n$, the ideal
	sequence and \cite[Proposition~III. 6.11A]{Har77}
	show that $\mathcal I_Z$ has local projective
	dimension at most $n-1$.
	Each sequence \eqref{eq:simple-evaluation}
	decreases any positive local projective dimension
	by one. Hence
	$$
	E:=G_{n-1}
	$$
	is locally free.
	
	\medskip
	\begin{step}{Simplicity and the quadratic obstruction.}\label{sssssstep3}
	\end{step}
	
	Denote by $D^b(X)$ the bounded derived category
	of coherent sheaves on $X$, and set
	$M_j:=\mathcal O_X(-m_j)$.
	Since $X$ is strict Calabi--Yau,
	$M_j\otimes K_X\simeq M_j$ and
	$$
	\operatorname{Ext}^q_X(M_j,M_j)
	\simeq H^q(X,\mathcal O_X)
	=
	\begin{cases}
		\mathbb C,&q=0,n,\\
		0,&q\neq0,n.
	\end{cases}
	$$
	Thus $M_j$ is spherical in the sense of Seidel--Thomas
	\cite[Definition~1.1]{ST01};
	see also \cite[Example~3.3]{ST01}.
	By \cite[Theorem~1.2]{ST01}, its Seidel--Thomas
	twist $T_{M_j}$ is an exact autoequivalence
	of $D^b(X)$.
	
		\medskip
	By \eqref{eq:simple-serre-vanishing},
	$$
	\operatorname{Hom}_X(M_j,G_{j-1})=V_j,
	\quad
	\operatorname{Ext}^q_X(M_j,G_{j-1})=0
	\,\,(q>0).
	$$
	The evaluation description of the twist
	\cite[p.~44]{ST01} therefore gives the triangle
	$$
	V_j\otimes M_j
	\longrightarrow G_{j-1}
	\longrightarrow T_{M_j}(G_{j-1})
	\longrightarrow (V_j\otimes M_j)[1].
	$$
	Comparing with \eqref{eq:simple-evaluation}, we obtain
	$$
	T_{M_j}(G_{j-1})\simeq G_j[1];
	$$
	see also \cite[Lemma~8.2 \& (8.1)]{JS12}.
	Applying these autoequivalences successively gives
	\begin{equation}\label{eq:simple-ext-identification}
		\operatorname{Ext}^{\bullet}_X(\mathcal I_Z,\mathcal I_Z)
		\simeq
		\operatorname{Ext}^{\bullet}_X(E,E),
	\end{equation}
	compatibly with composition.
	
		\medskip
	In degree zero, \eqref{eq:simple-ext-identification}
	shows that $E$ is simple.
	Let $\beta\in\operatorname{Ext}^1_X(E,E)$
	correspond to $\alpha$. Then
	\eqref{eq:simple-ideal-obstruction} gives
	\begin{equation}\label{eq:simple-bundle-obstruction}
		\beta\circ\beta\neq0
		\quad\text{in}\quad
		\operatorname{Ext}^2_X(E,E)
		\simeq H^2(X,\End E).
	\end{equation}
Since $H^2(X,\mathcal O_X)=0$, the trace decomposition gives
$$
H^2(X,\End^0E)\simeq H^2(X,\End E)\neq0.
$$
In particular, $\operatorname{rank}E\geq2$.

	\begin{step}{Persistence under joint deformations.}
	\end{step}
	
	Let $X_R$ be an arbitrary first-order deformation
	of $X$ over $R=\mathbb C[\epsilon]/(\epsilon^2)$.
	Since $H^1(X,\mathcal O_X)=H^2(X,\mathcal O_X)=0$,
	the same argument as in \eqref{thomas-picard}
	extends $\mathcal O_X(1)$ to $X_R$.
	Keeping the defining ideal constant in a local
	trivialization near $p$ also extends $Z$ to
	a flat subspace $Z_R\subset X_R$.
	
	Starting with $\mathcal I_{Z_R}$, we lift
	\eqref{eq:simple-evaluation} successively.
	The sections defining the evaluation maps lift
	by \eqref{eq:simple-serre-vanishing}, and the
	lifted maps remain surjective by Nakayama's lemma.
	Their kernels are $R$-flat, and the final kernel
	is locally free since its central fibre is $E$.
	Thus $E$ extends over every first-order deformation of $X$.
	
		\medskip
	In the cohomology sequence of the Atiyah extension,
	$$
	H^1(X,A(E))
	\longrightarrow H^1(X,T_X^{1,0})
	\longrightarrow H^2(X,\End E)
	\xrightarrow{\;\iota\;}H^2(X,A(E)),
	$$
	the first arrow is therefore surjective,
	so $\iota$ is injective.
	
	The class $\beta$ defines a first-order deformation
	of $(X,E)$ with $X$ fixed.
	Its quadratic obstruction is
	$-\iota(\beta\circ\beta)$, which is nonzero by
	\eqref{eq:simple-bundle-obstruction} and the
	injectivity of $\iota$.
	Hence $\operatorname{Kur}(X,E)$ is singular.
\end{proof}

\medskip
\medskip
\medskip
\medskip
\medskip
\medskip
\medskip
	
	\appendix
	\renewcommand{\thesection}{\Roman{section}}

	\section{Obstructed joint deformations of Hermitian flat
		bundles on complex tori (joint with Hisashi Kasuya)}\label{jw}
	\begin{center}
		\scriptsize HISASHI KASUYA
		\alphafootnote{Graduate School of Mathematics, Nagoya University, Nagoya, 
			Japan; Email address: \texttt{kasuya@math.nagoya-u.ac.jp}
		} \&  RUNZE ZHANG
	\end{center}
	
	\medskip

	\begin{theorem}\label{with Kasuya}
		Let $X$ be a complex torus of complex dimension $n\geq 2$, and let $L\in \mathrm{Pic}^0(X)$. Set 
		\[
		E= L^{\oplus r},\quad r\geq 2.
		\]
		Then the Kuranishi space $\operatorname{Kur}(X,E)$ of joint deformations
		of the pair $(X,E)$ is singular. 
	\end{theorem}
	
	\begin{proof}
		We divide the proof into two steps.

		\setcounter{step}{0}
		\begin{step}
			Reduction to the Kuranishi space of the bundle
		\end{step}
		Since $X$ is a complex torus, it is compact K\"ahler. Endow $X$
		with a flat translation-invariant K\"ahler metric. Since
		$L\in\operatorname{Pic}^0(X)$, the line bundle $L$ admits a
	flat Hermitian metric. Equip
		\(
		E=L^{\oplus r}
		\)
		with the induced flat metric. 
		The splitting theorem of Kasuya--Purho
		\cite[Proposition 1.1]{KP26} then gives an isomorphism of analytic germs
		\[
		\bigl(\operatorname{Kur}(X,E),0\bigr)
		\simeq
		\bigl(\operatorname{Kur}(X),0\bigr)
		\times
		\bigl(\operatorname{Kur}_X(E),0\bigr),
		\]
		where $\operatorname{Kur}_X(E)$ denotes the Kuranishi space of
		deformations of the holomorphic structure of $E$ with the complex
		structure of $X$ fixed.
		
		Since $X$ has trivial
		canonical bundle, it has unobstructed deformations  thanks to  the
		Bogomolov--Tian--Todorov theorem, it therefore suffices to prove that
		$\operatorname{Kur}_X(E)$ is singular at $0$.

		\begin{step}
			Identification with the germ of the commuting variety
		\end{step}
		Recall that the infinitesimal deformations of a holomorphic vector
		bundle $E$ over a fixed compact complex manifold $X$ are parametrized
		by $H^1(X,\operatorname{End}E)$, while the obstructions lie in
		$H^2(X,\operatorname{End}E)$. In our case,
		\(
		\operatorname{End}E
		\simeq
		\mathcal O_X\otimes_{\mathbb C}\mathfrak{gl}_r(\mathbb C).
		\)
		Therefore
		\[
		H^1(X,\operatorname{End}E)
		\simeq
		H^{0,1}(X)\otimes\mathfrak{gl}_r(\mathbb C)
		\simeq
		\bigoplus_{i=1}^n
		\mathbb C w_i\otimes\mathfrak{gl}_r(\mathbb C),
		\]
		where $
		\left\{w_1,\ldots,w_n\right\} 
		$
		denotes  a basis of translation-invariant harmonic $(0,1)$-forms.
		
		Denote by $E_{\alpha\beta}$, $1\leq\alpha,\beta\leq r$, the standard
		basis of $\mathfrak{gl}_r(\mathbb C)$. Then the first-order
		term in the Kuranishi construction can be written as
		\[
		\varphi_1(t)
		=
		\sum_{i=1}^n
		\sum_{\alpha,\beta=1}^r
		t_{i,\alpha\beta}\,
		w_i\otimes E_{\alpha\beta}:=\sum_{i=1}^n w_i\otimes A_i(t),
		\]
		where we define
		\(
		A_i(t)
		:=
		\sum_{\alpha,\beta=1}^r
		t_{i,\alpha\beta}E_{\alpha\beta}
		\in\mathfrak{gl}_r(\mathbb C).
		\)
		
		Since the $w_i$ are translation-invariant harmonic forms, their wedge
		products are again harmonic. In particular, the bracket
		\[
		\frac12[\varphi_1(t),\varphi_1(t)]
		=
		\sum_{1\leq i<j\leq n}
		w_i\wedge w_j\otimes[A_i(t),A_j(t)]
		\]
		is harmonic. Consequently, the Green operator annihilates this bracket,
		and all higher-order terms in the canonical Kuranishi series vanish.
		Thus the obstruction map is precisely
		\[
		\operatorname{Ob}_E(t)
		=
		\sum_{1\leq i<j\leq n}
		w_i\wedge w_j\otimes[A_i(t),A_j(t)].
		\]
		Furthermore, since the forms
		$\{w_i\wedge w_j\,|\, 1\leq i<j\leq n\}$
		are linearly independent,
		\[
		\operatorname{Ob}_E(t)=0
		\quad\Longleftrightarrow\quad
		[A_i(t),A_j(t)]=0
		\quad\text{for all }i<j.
		\]
		
		Define the commuting variety of $n$-tuples of $r\times r$ matrices by
		\[
		\mathcal C_{n,r}
		:=
		\left\{
		(A_1,\ldots,A_n)\in\mathfrak{gl}_r(\mathbb C)^n
		\;\middle|\;
		[A_i,A_j]=0
		\text{ for every }i<j
		\right\}.
		\]
		Then there is an isomorphism of analytic germs
		\[
		\bigl(\operatorname{Kur}_X(E),0\bigr)
		\simeq
		\bigl(\mathcal C_{n,r},0\bigr).
		\]
		
		\medskip
		
		It therefore suffices to consider the following two cases:
		\begin{enumerate}[{(i)}]
			\item If $\operatorname{Kur}_X(E)$ is reducible at $0$, then the origin
			cannot be a smooth point.
			
			\item If $\operatorname{Kur}_X(E)$ is irreducible at $0$, 
			then the analytic germ $(\mathcal C_{n,r},0)$ is irreducible.
			
			We then  claim that the underlying topological space
			$\lvert\mathcal C_{n,r}\rvert$ is globally irreducible. Suppose, to
			the contrary, that
			\[
			\lvert\mathcal C_{n,r}\rvert
			=
			Z_1\cup\cdots\cup Z_\ell,
			\quad \ell\geq2,
			\]
			is its decomposition into irreducible components.
			
			Since the defining commutator equations are homogeneous, every
			$Z_\alpha$ is a cone and therefore contains the origin. Set
			\(
			W:=Z_2\cup\cdots\cup Z_\ell.
			\)
			Then
			\[
			\bigl(\mathcal C_{n,r},0\bigr)
			=
			(Z_1,0)\cup(W,0).
			\]
			
			Both $(Z_1,0)$ and $(W,0)$ are proper analytic subgerms. Indeed,
			suppose for instance that
			\(
			(Z_1,0)=(\mathcal C_{n,r},0).
			\)
			Then $Z_1$ agrees with $\lvert\mathcal C_{n,r}\rvert$ in some
			neighborhood $U$ of the origin. Given any
			$z\in\lvert\mathcal C_{n,r}\rvert$, choose
			$\lambda\in\mathbb C^*$ sufficiently small that $\lambda z\in U$.
			Then $\lambda z\in Z_1$. Since $Z_1$ is a cone, it follows that
			\[
			z=\lambda^{-1}(\lambda z)\in Z_1.
			\]
			Thus $\lvert\mathcal C_{n,r}\rvert=Z_1$, contradicting
			$\ell\geq2$. The same argument applies to $W$.
			
			Hence $(\mathcal C_{n,r},0)$ is the union of two proper closed
			analytic subgerms, contradicting its irreducibility. Therefore
			$\ell=1$, and $\lvert\mathcal C_{n,r}\rvert$ is globally
			irreducible.
			
			\medskip
			Applying \cite[Proposition 6]{GS00}, or equivalently
			\cite[Lemma 7.9.1]{OCV11}, we obtain\footnote{We learned of these
				references from Alexander Stasinski's helpful answer to the
				MathOverflow question
				\href{https://mathoverflow.net/questions/499942/dimension-of-the-variety-of-commuting-matrices}
				{\emph{Dimension of the variety of commuting matrices}}.}
			\[
			\dim\mathcal C_{n,r}
			=
			r^2+(n-1)r,
			\]
			and hence the local analytic dimension of $\textrm{Kur}_X(E)$ at $0$ is 
			\[
			\dim_0\operatorname{Kur}_X(E)
			=
			r^2+(n-1)r.
			\]
			On the other hand, the dimension of the Zariski tangent space at
			$0$ is
			\[
			\dim_{\mathbb C}T_0\operatorname{Kur}_X(E)
			=
			h^1(X,\operatorname{End}E)
			=
			nr^2.
			\]
			
			For $n,r\geq2$, we have
			\[
			\begin{aligned}
				\dim_{\mathbb C}T_0\operatorname{Kur}_X(E)
				-
				\dim_0\operatorname{Kur}_X(E)
				&=
				nr^2-\bigl(r^2+(n-1)r\bigr)\\
				&=
				(n-1)r(r-1)\\
				&>0.
			\end{aligned}
			\]
		\end{enumerate}
		
		In either case, $\operatorname{Kur}_X(E)$, and hence
		$\operatorname{Kur}(X,E)$, is singular at $0$. This completes the proof.
	\end{proof}
	
	\begin{remark} As in Remark~\ref{thomas} and
		Step~\ref{sssssstep3} of the proof of
		Theorem~\ref{thm:simple-obstructed-pairs},
		the bundles in Theorem~\ref{with Kasuya}
		do not satisfy the vanishing assumption
		$H^2(X,\End^0E)=0$ in
		Theorems~\ref{main thm} and~\ref{main thm2}.
		Indeed, since $E=L^{\oplus r}$,
		$$
		H^2(X,\End^0E)
		\simeq
		H^2(X,\mathcal O_X)\otimes\mathfrak{sl}_r(\mathbb C),
		$$
		where $\mathfrak{sl}_r(\mathbb C)$ denotes the space
		of trace-free $r\times r$ complex matrices.
		Thus
		$$
		h^2(X,\End^0E)
		=
		\binom{n}{2}(r^2-1)>0
		\quad(n,r\geq2).
		$$
	\end{remark}
	
	\begin{remark}\label{rem:wan-huybrechts}
			Huybrechts \cite[\S2]{Huy95} asked whether the tangent
			bundle of a strict Calabi--Yau threefold can have
			obstructed deformations with the threefold fixed.
			Wan \cite[Theorem~1.5]{Wan26} constructs such a
			projective threefold $Z$ with Hermitian flat
			tangent bundle.
			The splitting theorem
			\cite[Proposition~1.1]{KP26} then shows that the
			joint deformations of $(Z,T_Z^{1,0})$ are also
			obstructed, thereby also answering
			Felten's Question~\ref{q2} positively.
		\end{remark}

	
	\medskip
	\medskip
	\medskip

 \noindent{\bfseries Acknowledgements.}
		The author is deeply grateful to his postdoctoral mentors, Professors
		Kwokwai Chan and Naichung Conan Leung, for their constant
		encouragement and support.
		He also thanks Professors  Daniele Angella, Kwokwai Chan, Simon Felten,
		Osamu Fujino, Naichung Conan Leung, Kefeng Liu,   Lingxu Meng, Sheng Rao,
		Xueyuan Wan, Xiangdong Yang, and Kang Zuo
		for many helpful discussions and comments.
		
		The author especially thanks Professor Hisashi Kasuya for many
	fruitful discussions via email, which led
		to their collaboration in Appendix~\ref{jw}.
		He is also particularly grateful to Professor Richard P. 
		Thomas for his comments on an earlier draft and for
		communicating Duco van Straten's suggestion of using
		an obstructed zero-dimensional subscheme and a locally
		free resolution. This suggestion led to the construction
		in Theorem~\ref{thm:simple-obstructed-pairs}.
	
	\medskip
	\noindent\textbf{Declaration on the use of AI.} \,\,
This paper is a continuation of \cite{Zha26}.
Most of its results were presented by the author
at conferences during the 2025--2026 academic year.
The author constructed the example in
Proposition~\ref{prop:weak-ddbar-threefold}, for which
$H^2(X,A(E))=0$ whenever $H^2(X,\End^0E)=0$.
The example in
Proposition~\ref{prop:nonzero-obstruction-example}
was suggested by ChatGPT.
It has a nonzero obstruction space, and the forgetful
morphism
$\operatorname{Def}(X,E)\to\operatorname{Def}(X)$
is not smooth.
The author checked the construction and verified
all details of the proof.
ChatGPT was also used to  assist with English grammar and language editing.
The author takes full responsibility for the mathematical
content and final text.

\end{document}